\documentclass[a4paper, 10pt, DIV10, headinclude=false, footinclude=false, leqno]{scrartcl}

\usepackage{amsmath, amssymb, amsfonts}
\usepackage[amsmath, amsthm, thmmarks]{ntheorem}
\usepackage[utf8]{inputenc}
\usepackage[english]{babel}
\usepackage{url}
\usepackage{mathtools}
\usepackage{esint}

\usepackage{tikz}

\numberwithin{figure}{section}
\numberwithin{table}{section}
\numberwithin{equation}{section}

\newenvironment{abstr}[1]{ \vspace{.05in}\footnotesize
       \parindent .2in
         {\upshape\bfseries #1. }\ignorespaces}{\par\vspace{.1in}}
\newenvironment{Abstract}{\begin{abstr}{Abstract}}{\end{abstr}}
\newenvironment{keywords}{\begin{abstr}{Key words}}{\end{abstr}}
\newenvironment{AMS}{\begin{abstr}{AMS subject classifications}}{\end{abstr}}

\newtheorem{theorem}{Theorem}[section]
\newtheorem{lemma}[theorem]{Lemma}

\newtheorem{proposition}[theorem]{Proposition}
\newtheorem{conclusion}[theorem]{Conclusion}
\theoremstyle{definition}
\newtheorem{definition}[theorem]{Definition}
\newtheorem{assumption}[theorem]{Assumption}
\newtheorem{remark}[theorem]{Remark}

 \newcommand{\quotes}[1]{``#1''}

\DeclareMathOperator{\supp}{supp}
\DeclareMathOperator{\Int}{int}
\DeclareMathOperator{\diam}{diam}
\DeclareMathOperator{\Div}{div}
\DeclareMathOperator{\curl}{curl}

\DeclareMathOperator{\ol}{ol}

\DeclareMathOperator{\ms}{ms}
\newcommand{\nz}{\mathbb{N}}       % natural numbers
\newcommand{\rz}{\mathbb{R}}       % real numbers
\newcommand{\cz}{\mathbb{C}}       % complex numbers
\newcommand{\pre}{\mathrm{pre}}

\newcommand\Va{\mathbf{a}}
\newcommand\Vb{\mathbf{b}}

\newcommand\Vf{\mathbf{f}}

\newcommand\Vn{\mathbf{n}}

\newcommand\Vt{\mathbf{t}}
\newcommand\Vv{\mathbf{v}}
\newcommand\Vu{\mathbf{u}}
\newcommand\Vw{\mathbf{w}}
\newcommand\Vx{\mathbf{x}}

\newcommand\Vz{\mathbf{z}}

\newcommand\VE{\mathbf{E}}

\newcommand\VH{\mathbf{H}}

\newcommand\VW{\mathbf{W}}
\newcommand\Vpsi{\boldsymbol{\psi}}
\newcommand\Vphi{\boldsymbol{\phi}}
\newcommand\Vtau{\boldsymbol{\tau}}

\newcommand\CB{\mathcal{B}}

\newcommand\CL{\mathcal{L}}

\newcommand\CN{\mathcal{N}}

\newcommand\CR{\mathcal{R}}
\newcommand\CS{\mathcal{S}}
\newcommand\CT{\mathcal{T}}

\newcommand{\Kf}{\mathcal{K}} %\newcommand{\Kf}{\CC}
\newcommand{\Gf}{\mathcal{G}} 

\newcommand\UN{\textup{N}}

\newcommand\id{\operatorname{id}}

\usepackage[textsize=footnotesize]{todonotes}
\usepackage[xcolor=dvipdf]{changes}
\setremarkmarkup{\todo[color=Changes@Color#1!40]{\textcolor{black}{#1:{#2}}}}

\allowdisplaybreaks[4]

\begin{document}

\title{Numerical homogenization of H(curl)-problems}

\author{Dietmar Gallistl\footnotemark[2] \and Patrick Henning\footnotemark[3]\and Barbara Verf\"urth\footnotemark[4]}
\date{}
\maketitle

\renewcommand{\thefootnote}{\fnsymbol{footnote}}
\footnotetext[2]{Institut f\"ur Angewandte und Numerische Mathematik, Karlsruher Institut f\"ur Technologie, Englerstr. 2, D-76131 Karlsruhe, Germany}
\footnotetext[3]{Department of Mathematics, KTH Royal Institute of Technology, Lindstedtsv\"agen 25, SE-100 44 Stockholm, Sweden}
\footnotetext[4]{Applied Mathematics, Westf\"alische Wilhelms-Uni\-ver\-si\-t\"at M\"unster, Einsteinstr. 62, D-48149 M\"unster, Germany}
\renewcommand{\thefootnote}{\arabic{footnote}}

\begin{Abstract}
If an elliptic differential operator associated with an 
$\mathbf{H}(\mathrm{curl})$-problem involves rough (rapidly varying) coefficients,
then solutions to the corresponding $\mathbf{H}(\mathrm{curl})$-problem admit typically
very low regularity, which leads to arbitrarily bad convergence rates
for conventional numerical schemes.
The goal of this paper is to show that the missing regularity can be 
compensated through a corrector operator. More precisely, we consider 
the lowest order N{\'e}d{\'e}lec finite element space 
and show the existence of a linear corrector operator 
with four central properties: 
it is computable, $\mathbf{H}(\mathrm{curl})$-stable, 
quasi-local and allows for a correction of coarse finite element 
functions so that first-order estimates (in terms of the coarse 
mesh-size) in the $\mathbf{H}(\mathrm{curl})$ norm are obtained provided 
the right-hand side belongs to $\mathbf{H}(\mathrm{div})$.
With these four properties, a practical application is to construct 
generalized finite element spaces which can be straightforwardly used 
in a Galerkin method. In particular, this characterizes a homogenized 
solution and a first order corrector, including corresponding 
quantitative error estimates without the requirement of scale separation.
\end{Abstract}



\section{Introduction}
Electromagnetic wave propagation plays an essential role in many physical applications, for instance, in the large field of wave optics.
In the last years, multiscale and heterogeneous materials are studied with great interest, e.g., in the context of photonic crystals \cite{JJWM08phc}.
These materials can exhibit unusual and astonishing (optical) properties, such as band gaps, perfect transmission or negative refraction \cite{CJJP02negrefraction, EP04negphC, LS15negindex}.
These problems are modeled by Maxwell's equations, which involve the curl-operator and the associated Sobolev space $\VH(\curl)$.
Additionally, the coefficients in the problems are rapidly oscillating on a fine scale for the context of photonic crystals and metamaterials.
The numerical simulation and approximation of the solution is then a challenging task for the following three reasons.
1.\ As with all multiscale problems, a direct treatment with standard methods in infeasible in many cases because it needs grids which resolve all discontinuities or oscillations of the material parameters.
2.\ Solutions to $\VH(\curl)$-problems with discontinuous coefficients in Lip\-schitz domains can have arbitrarily low regularity, see \cite{BGL13regularitymaxwell, CDN99maxwellinterface, Cost90regmaxwellremark}.
Hence, standard methods (see e.g., \cite{Monk} for an overview) suffer from bad convergence rates and fine meshes are needed to have a tolerably small error.
3.\ Due to the large kernel of the curl-operator, we cannot expect that the $L^2$-norm is of a lower order as the full $\VH(\curl)$-norm (the energy norm).
Thus, it is necessary to consider dual norms or the Helmholtz decomposition to obtain improved a priori error estimates.

In order to deal with the rapidly oscillating material parameters, we consider multiscale methods and thereby aim at a feasible numerical simulation.
In general, these methods try to decompose the exact solution into a macroscopic contribution (without oscillations), which can be discretized on a coarse mesh, and a fine-scale contribution.
Analytical homogenization for locally periodic $\VH(\curl)$-problems shows that there exists such a decomposition, where the macroscopic part is a good approximation in $H^{-1}$ and an additional fine-scale corrector leads to a good approximation in $L^2$ and $\VH(\curl)$, cf.\ \cite{CH15homerrormaxwell, HOV15maxwellHMM, Well2}.
Based on these analytical results, multiscale methods are developed, e.g., the Heterogeneous Multiscale Method in \cite{HOV15maxwellHMM, CFS17hmmmaxwell} and asymptotic expansion methods in \cite{CZAL10maxwell}.
The question is now in how far such considerations can be extended beyond the (locally) periodic case.

The main contribution of this paper is the numerical homogenization of $\VH(\curl)$-elliptic problems -- beyond the periodic case and without assuming scale separation.
The main findings can be summarized as follows.
We show that the exact solution can indeed be decomposed into a coarse and fine part, using a suitable interpolation operator.
The coarse part gives an optimal approximation in the $H^{-1}$-norm, the best we can hope for in this situation.
In order to obtain optimal $L^2$ and $\VH(\curl)$ approximations, we have to add a so called fine-scale corrector or corrector Green's operator.
This corrector shows exponential decay and can therefore be truncated to local patches of macroscopic elements, so that it can be computed efficiently.

This technique of numerical homogenization is known as Localized Orthogonal Decomposition (LOD) and it was originally proposed by M{\aa}lqvist and Peterseim \cite{MP14LOD} to solve elliptic multiscale problems 
through an orthogonalization procedure with a problem-specific \quotes{multiscale} inner product.
The LOD has been extensively studied in the context of Lagrange finite elements \cite{HM14LODbdry, HP13oversampl}, where we particularly refer to the contributions written on wave phenomena \cite{AH17LODwaves, BrG16, bgp2017, GP15scatteringPG, OV16a, P15LODhelmholtz, PeS17}. Aside from Lagrange finite elements, an LOD application in Raviart-Thomas spaces was given in \cite{HHM16LODmixed}. 

A crucial ingredient for numerical homogenization procedures in the spirit of LODs is the choice of a suitable interpolation operator.
As we will see later, in our case we require it to be computable, $\VH(\curl)$-stable, (quasi-)local and that it commutes with the curl-operator.
Constructing an operator that enjoys such properties is a very delicate task and a lot of operators have been suggested -- with different backgrounds and applications in mind.
The nodal interpolation operator, see e.g.\ \cite[Thm.\ 5.41]{Monk}, and the interpolation operators introduced in \cite{DB05maxwellpintpol} are not well-defined on $\VH(\curl)$ and hence lack the required stability.
Various (quasi)-interpolation operators are constructed as composition of smoothing and some (nodal) interpolation, such as
\cite{Chr07intpol, CW08intpol, DH14aposteriorimaxwell, EG15intpol, Sch05multilevel,Sch08aposteriori}.
For all of them, the kernel of the operator is practically hard or even impossible to compute and they only fulfill the projection \emph{or} the locality property.
Finally, we mention the interpolation operator of \cite{EG15intpolbestapprox} which is local and a projection, however, which does not commute with the exterior derivative.
A suitable candidate (and to the authors' best knowledge, the only one) that enjoys all required properties was proposed by Falk and Winther in \cite{FalkWinther2014}.

This paper thereby also shows the applicability of the Falk-Winther operator. 
In this context, we mention two results, which may be of own interest: a localized regular decomposition of the interpolation error (in the spirit of \cite{Sch08aposteriori}), and the practicable implementation of the Falk-Winther operator as a matrix.
The last point admits the efficient implementation of our numerical scheme and we refer to \cite{EHMP16LODimpl} for general considerations.

The paper is organized as follows.
Section \ref{sec:setting} introduces the general curl-curl-problem under consideration and briefly mentions its relation to Maxwell's equations.
In Section \ref{sec:motivation}, we give a short motivation of our approach from periodic homogenization.
Section \ref{sec:intpol} introduces the necessary notation for meshes, finite element spaces, and interpolation operators.
We introduce the Corrector Green's Operator in Section \ref{sec:LODideal} and show its approximation properties.
We localize the corrector operator in Section \ref{sec:LOD} and present the main apriori error estimates.
The proofs of the decay of the correctors are given in Section \ref{sec:decaycorrectors}.
Details on the definition of the interpolation operator and its implementation are given in Section \ref{sec:intpolimpl}.

The notation $a\lesssim b$ is used for $a\leq Cb$ with a constant $C$ independent of the mesh size $H$ and the oversampling parameter $m$.
It will be used in (technical) proofs for simplicity and readability.

\section{Model problem}
\label{sec:setting}
Let $\Omega\subset \rz^3$ be an open, bounded, contractible 
domain with polyhedral Lipschitz boundary.
We consider the following so called curl-curl-problem: Find $\Vu:\Omega\to\cz^3$ such that
\begin{equation}
\label{eq:curlcurl}
\begin{split}
\curl(\mu\curl\Vu)+\kappa\Vu&=\Vf\quad\text{in }\Omega,\\
\Vu\times \Vn&=0\quad\text{on }\partial \Omega
\end{split}
\end{equation}
with the outer unit normal $\Vn$ of $\Omega$. 
Exact assumptions on the parameters $\mu$ and $\kappa$ and the right-hand side $\Vf$ are given in Assumption~\ref{asspt:sesquiform} below,
but we implicitly assume that the above problem is a multiscale problem, i.e.\ the coefficients $\mu$ and $\kappa$ are rapidly varying on a very fine sale. 

Such curl-curl-problems arise in various formulations and reductions of Maxwell's equations and we shortly give a few examples.
In all cases, our coefficient $\mu$ equals $\tilde{\mu}^{-1}$ with the magnetic permeability $\tilde{\mu}$, a material parameter.
The right-hand side $\Vf$ is related to (source) current densities.
One possible example are Maxwell's equations in a linear conductive medium, subject to Ohm's law, together with the so called time-harmonic ansatz $\hat{\Vpsi}(x,t)=\Vpsi(x)\exp(-i\omega t)$ for all fields. 
In this case, one obtains the above curl-curl-problem with $\Vu=\VE$, the electric field, and $\kappa=i\omega\sigma-\omega^2\varepsilon$ related to the electric permittivity $\varepsilon$ and  the conductivity $\sigma$ of the material.
Another example are implicit time-step discretizations of eddy current simulations, where the above curl-curl-problem has to be solved in each time step.
In that case $\Vu$ is the vector potential associated with the magnetic field and $\kappa\approx\sigma/\tau$, where $\tau$ is the time-step size. Coefficients with multiscale properties can for instance arise in the context of photonic crystals.

Before we define the variational problem associated with our general curl-curl-problem \eqref{eq:curlcurl}, we need to introduce some function spaces.
In the following, bold face letters will indicate vector-valued quantities and all functions are complex-valued, unless explicitly mentioned.
For any bounded subdomain $G\subset \Omega$, we define the space
\[\VH(\curl, G):=\{ \Vv\in L^2(G, \cz^3)|\curl\Vv\in L^2(G, \cz^3)\}\]
with the inner product $(\Vv, \Vw)_{\VH(\curl, G)}:=(\curl\Vv, \curl\Vw)_{L^2(G)}+(\Vv, \Vw)_{L^2(G)}$ with the complex $L^2$-inner product. 
We will omit the domain $G$ if it is equal to the full domain $\Omega$.
The restriction of $\VH(\curl, \Omega)$ to functions with a zero tangential trace 
is defined as
\[\VH_0(\curl, \Omega):=\{\Vv\in \VH(\curl, \Omega)|\hspace{3pt} \Vv\times \Vn \vert_{\partial \Omega} =0\}. \]
Similarly, we define the space 
\[\VH(\Div, G):=\{\Vv\in L^2(G, \cz^3)|\Div \Vv\in L^2(G, \cz)\}\]
with corresponding inner product $(\cdot, \cdot)_{\VH(\Div, G)}$.
For more details we refer to \cite{Monk}. 

We make the following assumptions on the data of our problem. 
\begin{assumption}
\label{asspt:sesquiform}
Let $\Vf\in \VH(\Div, \Omega)$ and let $\mu\in L^\infty(\Omega, \rz^{3 \times 3})$ and $\kappa\in L^\infty(\Omega, \cz^{3 \times 3})$.
For any open subset $G\subset\Omega$,
we define the sesquilinear  form 
$\CB_{G}: \VH(\curl,G)\times \VH(\curl,G)\to \cz$ as
\begin{equation}
\label{eq:sesquiform}
\CB_{G}(\Vv, \Vpsi):=(\mu\curl \Vv, \curl\Vpsi)_{L^2(G)}
                      +(\kappa\Vv, \Vpsi)_{L^2(G)},
\end{equation}
and set $\CB:=\CB_\Omega$.
The form $\CB_{G}$ is obviously continuous, i.e.\ there is $C_B>0$ such that 
\begin{equation*}
 |\CB_{G}(\Vv, \Vpsi)|\leq C_B\|\Vv\|_{\VH(\curl,G)}\|\Vpsi\|_{\VH(\curl,G)}
 \quad\text{for all }\Vv,\Vpsi\in\VH(\curl,G).
\end{equation*}
We furthermore assume that $\mu$ and $\kappa$ are such that 
$\CB: \VH_0(\curl)\times \VH_0(\curl)\to \cz$
is $\VH_0(\curl)$-elliptic, 
i.e.\ there is $\alpha>0$ such that
\[
 |\CB(\Vv, \Vv)|\geq \alpha\|\Vv\|^2_{\VH(\curl)}
 \quad\text{for all }\Vv\in\VH_0(\curl) .
\]
\end{assumption}

We now give a precise definition of our model problem for this article.
Let Assumption \ref{asspt:sesquiform} be fulfilled. We look for $\Vu\in \VH_0(\curl, \Omega)$ such that
\begin{equation}
\label{eq:problem}
\CB(\Vu, \Vpsi)=(\Vf, \Vpsi)_{L^2(\Omega)} \quad\text{for all } \Vpsi\in \VH_0(\curl, \Omega).
\end{equation}
Existence and uniqueness of a solution to \eqref{eq:problem} follow from the Lax-Milgram-Babu{\v{s}}ka theorem  \cite{Bab70fem}.
Assumption \ref{asspt:sesquiform} is fulfilled in the following two important examples mentioned at the beginning:
(i) a strictly positive real function in the identity term, i.e.\ $\kappa\in L^\infty(\Omega, \rz)$, as it occurs in the time-step discretization of eddy-current problems;
(ii) a complex $\kappa$ with strictly negative real part and strictly positive imaginary part, as it occurs for time-harmonic Maxwell's equations in a conductive medium.
Further possibilities of $\mu$ and $\kappa$ yielding an $\VH(\curl)$-elliptic problem are described in \cite{FR05maxwell}.

\begin{remark}
The assumption of contractibility of $\Omega$ is only required to ensure the existence of local regular decompositions later used in the proof of Lemma \ref{lem:localregulardecomp}. We note that this assumption can be relaxed by assuming that $\Omega$ is simply connected in certain local subdomains formed by unions of tetrahedra (i.e. in patches of the form $\UN(\Omega_P)$, using the notation from Lemma \ref{lem:localregulardecomp}).
\end{remark}

\section{Motivation of the approach}
\label{sec:motivation}

For the sake of the argument, let us consider model problem \eqref{eq:curlcurl} for the case that the coefficients $\mu$ and $\kappa$ are replaced by parametrized multiscale coefficients $\mu_{\delta}$ and $\kappa_\delta$, respectively.
Here, $0<\delta \ll 1$ is a small parameter that characterizes the roughness of the coefficient or respectively the speed of the variations, i.e.\ the smaller $\delta$, the faster the oscillations of $\mu_{\delta}$ and $\kappa_\delta$. 
If we discretize this model problem
in the lowest order N{\'e}d{\'e}lec finite element space $\mathring{\CN}(\CT_H)$, we have the classical error estimate of the form
\begin{align*}
\inf_{\mathbf{v}_H \in \mathring{\CN}(\CT_H)} \| \Vu_{\delta} - \mathbf{v}_H \|_{\VH(\curl)} \le C H \left( \| \Vu_{\delta} \|_{H^1(\Omega)} + \| \curl \Vu_{\delta} \|_{H^1(\Omega)} \right).
\end{align*}
However, if the coefficients $\mu_{\delta}$ and $\kappa_\delta$ are discontinuous the necessary regularity for this estimate is not available, see \cite{Cost90regmaxwellremark, CDN99maxwellinterface, BGL13regularitymaxwell}. 
On the other hand, if $\mu_{\delta}$ and $\kappa_\delta$ are sufficiently regular but $\delta$ small, then we face the blow-up with $\| \Vu_{\delta}  \|_{H^1(\Omega)} + \| \curl \Vu_{\delta} \|_{H^1(\Omega)}\rightarrow \infty$ for $\delta \rightarrow 0$, which makes the estimate useless in practice, unless the mesh size $H$ becomes very small to compensate for the blow-up. This does not change if we replace the $\VH(\curl)$-norm by the $L^2(\Omega)$-norm since both norms are equivalent in our setting.

To understand if there exist any meaningful approximations of $\Vu_{\delta}$ in $\mathring{\CN}(\CT_H)$ (even on coarse meshes), we make a short excursus to classical homogenization theory. For that we assume that the coefficients $\mu_{\delta}(x)=\mu(x/\delta)$ and $\kappa_\delta(x)=\kappa(x/\delta)$ are periodically oscillating with period $\delta$. In this case it is known (cf.\ \cite{CFS17hmmmaxwell, HOV15maxwellHMM, Well2}) that the sequence of exact solutions $\Vu_{\delta}$ converges weakly in  $\VH_0(\curl)$ to a \emph{homogenized} function $\Vu_{0}$. Since $\Vu_0 \in \VH_0(\curl)$ is $\delta$-independent and slow, it can be well approximated in $\mathring{\CN}(\CT_H)$. Furthermore, there exists a \emph{corrector} $\Kf_{\delta}(\Vu_0)$ such that 
\[\Vu_{\delta} \approx (\id + \Kf_{\delta})\Vu_0 \]
is a good approximation in $\VH(\curl)$, i.e.\ the error converges strongly to zero with
\[
\| \Vu_{\delta} -( \Vu_0 + \Kf_{\delta}(\Vu_0)) \|_{\VH(\curl)} \rightarrow 0
\qquad \mbox{for } \delta \rightarrow 0.
\]
Here the nature of the corrector is revealed by two estimates. In fact, $\Kf_{\delta}(\Vu_0)$ admits a decomposition into a gradient part and a part with small amplitude (cf. \cite{HOV15maxwellHMM, CH15homerrormaxwell, Well2}) such that
\[
 \Kf_{\delta}(\Vu_0) =  \Vz_{\delta} + \nabla \theta_{\delta}
\]
with
\begin{align}
\label{hom-corrector-est-1}
\delta^{-1} \| \Vz_{\delta} \|_{L^2(\Omega)} + \| \Vz_{\delta} \|_{\VH(\curl)} &\le C\| \Vu_0 \|_{\VH(\curl)}\\
\label{hom-corrector-est-2}
\text{and}\qquad\delta^{-1} \| \theta_{\delta} \|_{L^2(\Omega)} + \| \nabla \theta_{\delta} \|_{L^2(\Omega)} &\le C \| \Vu_0 \|_{\VH(\curl)},
\end{align}
where $C=C(\alpha,C_B)$ only depends on the constants appearing in Assumption \ref{asspt:sesquiform}. First, we immediately see that the estimates imply that $\Kf_{\delta}(\Vu_0)$ is $\VH(\curl)$-stable in the sense that it holds
\begin{align*}
\| \Kf_{\delta}(\Vu_0) \|_{\VH(\curl)} \le C \| \Vu_0 \|_{\VH(\curl)}.
\end{align*}
Second, and more interestingly, we see that alone from the above properties, we can conclude that $\Vu_0$ \emph{must} be a good approximation of the exact solution in the space $H^{-1}(\Omega,\cz^3)$. In fact, using \eqref{hom-corrector-est-1} and \eqref{hom-corrector-est-2} we have for any
$\mathbf{v}\in H^1_0(\Omega,\cz^3)$ with $\| \mathbf{v} \|_{H^1(\Omega)}=1$ that
\begin{align*}
\left|\int_{\Omega} \Kf_{\delta}(\Vu_0) \cdot \mathbf{v} \right|= 
\left|\int_{\Omega} \Vz_{\delta} \cdot \mathbf{v} - \int_{\Omega} \theta_{\delta} \hspace{2pt} (\nabla \cdot \mathbf{v}) \right| \le 
\| \Vz_{\delta} \|_{L^2(\Omega)} + \| \theta_{\delta} \|_{L^2(\Omega)} 
\le C \delta \| \Vu_0 \|_{\VH(\curl)}.
\end{align*}
Consequently we have strong convergence in $H^{-1}(\Omega)$ with
\begin{align*}
\| \Vu_{\delta} - \Vu_0 \|_{H^{-1}(\Omega)} 
\le \| \Vu_{\delta} - ( \Vu_0 + \Kf_{\delta}(\Vu_0))\|_{H^{-1}(\Omega)} +  \| \Kf_{\delta}(\Vu_0) \|_{H^{-1}(\Omega)} \overset{\delta \rightarrow 0}{\longrightarrow} 0.
\end{align*}
We conclude two things. Firstly, even though the coarse space $\mathring{\CN}(\CT_H)$ does not contain good $\VH(\curl)$- or $L^2$-approximations, it still contains meaningful approximations in $H^{-1}(\Omega)$. 
Secondly, the fact that the 
coarse part
$\Vu_0$ 
is a good $H^{-1}$-approximation of $\Vu_{\delta}$ is an intrinsic conclusion from the properties of the correction $\Kf_{\delta}(\Vu_0)$.

In this paper we are concerned with the question if the above considerations can be transferred to a discrete setting beyond the assumption of periodicity. More precisely, defining a coarse level of resolution through the space $\mathring{\CN}(\CT_H)$, we ask if it is possible to find a coarse function $\Vu_H$ and an (efficiently computable) $\VH(\curl)$-stable operator $\Kf$, such that
\begin{align}
\label{motivation:int-estimates}
\| \Vu_{\delta} - \Vu_H \|_{H^{-1}(\Omega)} \le C H \qquad \mbox{and} \qquad \| \Vu_{\delta} - (I+\Kf)\Vu_H \|_{\VH(\curl)} \le CH,
\end{align}
with $C$ being independent of the oscillations in terms of $\delta$. A natural ansatz for the coarse part is $\Vu_H=\pi_H( \Vu_{\delta} )$ for a suitable projection $\pi_H : \VH(\curl) \rightarrow \mathring{\CN}(\CT_H)$. However, from the considerations above we know that $\Vu_H=\pi_H( \Vu_{\delta} )$ can only be a good $H^{-1}$-approximation if the error fulfills a discrete analog to the estimates \eqref{hom-corrector-est-1} and \eqref{hom-corrector-est-2}. Since $\Vu_{\delta} - \pi_H( \Vu_{\delta} )$ is nothing but an interpolation error, we can immediately derive a sufficient condition for our choice of $\pi_H$: we need that, for any $\Vv\in \VH_0(\curl, \Omega)$, there are $\Vz\in \VH^1_0(\Omega)$ and $\theta\in H^1_0(\Omega)$ such that
\[\Vv-\pi_H \Vv=\Vz+\nabla \theta\]
and 
\begin{equation}
\label{motivation:properties-pi-H}
\begin{split}
H^{-1}\|\Vz\|_{L^2(\Omega)}+\|\nabla \Vz\|_{\VH(\curl)} &\leq C \|\curl\Vv\|_{L^2(\Omega)},\\
H^{-1}\|\theta\|_{L^2(\Omega)}+\|\nabla \theta\|_{L^2(\Omega)}&\leq C \|\curl\Vv\|_{L^2(\Omega)}.
\end{split}
\end{equation}
This is a sufficient condition for $\pi_H$. Note that the above properties are not fulfilled for e.g. the $L^2$-projection. This resembles the fact that the $L^2$-projection does typically not yield a good $H^{-1}$-approximation in our setting.

We conclude this paragraph by summarizing that if we have a projection $\pi_H$ fulfilling \eqref{motivation:properties-pi-H}, then we can define a coarse scale numerically through the space 
$\mathring{\CN}(\CT_H) = \mbox{im}(\pi_H)$.  
On the other hand, to ensure that the corrector inherits the desired decomposition with estimates \eqref{motivation:int-estimates}, it needs to be constructed such that it maps into the kernel of the projection operator, i.e. $\mbox{im}(\Kf)\subset\mbox{ker}(\pi_H)$.

\section{Mesh and interpolation operator}
\label{sec:intpol}

In this section we introduce the basic notation for establishing our coarse scale discretization and we will present a projection operator that fulfills the sufficient conditions derived in the previous section.

Let $\CT_H$ be a regular partition of $\Omega$ into tetrahedra, such that $\cup\CT_H=\overline{\Omega}$ and any two distinct $T, T'\in \CT_H$ are either disjoint or share a common vertex, edge or face.
We assume the partition $\CT_H$ to be shape-regular and quasi-uniform.
The global mesh size is defined as $H:=\max\{ \diam(T)|T\in \CT_{H}\}$.
$\CT_H$ is a coarse mesh in the sense that it does not resolve the fine-scale oscillations of the parameters.

Given any (possibly even not connected) subdomain $G\subset \overline{\Omega}$ define its neighborhood via
\[\UN(G):=\Int(\cup\{T\in \CT_{H}|T\cap\overline{G}\neq \emptyset\})\]
and for any $m\geq 2$ the patches
\[\UN^1(G):=\UN(G)\qquad \text{and}\qquad\UN^m(G):=\UN(\UN^{m-1}(G)),\]
see Figure \ref{fig:patch} for an example.
The shape regularity implies that there is a uniform bound $C_{\ol, m}$ on the number of elements in the $m$-th order patch
\[\max_{T\in \CT_{H}}\operatorname{card}\{K\in \CT_{H}|K\subset\overline{\UN^m(T)}\}\leq C_{\ol, m}\]
and the quasi-uniformity implies that $C_{\ol, m}$ depends polynomially on  $m$.
We abbreviate $C_{\ol}:=C_{\ol, 1}$.

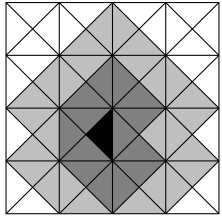
\begin{figure}
\begin{center}
\begin{tikzpicture}[scale=.7]
\path[fill=lightgray] (3,0)--(4,1)--(3.5,1.5)--(4,2)--(2,4)--(0,2)--(0.5, 1.5)--(0,1)--(1,0)--cycle;
\path[fill=gray] (2,0)--(3,1)--(2.5, 1.5)--(3,2)--(2,3)--(1,2)--(1,1)--cycle;
\path[fill=black] (1.5, 1.5)--(2,2)--(2,1)--cycle;
\draw(0,0) rectangle (4,4);
\draw(1,0)--(1,4);
\draw(2,0)--(2,4);
\draw(3,0)--(3,4);
\draw(0,1)--(4,1);
\draw(0,2)--(4,2);
\draw(0,3)--(4,3);
\draw (0,0)--(4,4);
\draw (0,1)--(3,4);
\draw(0,2)--(2,4);
\draw(0,3)--(1,4);
\draw(1,0)--(4,3);
\draw(2,0)--(4,2);
\draw(3,0)--(4,1);
\draw(0,4)--(4,0);
\draw(0,3)--(3,0);
\draw(0,2)--(2,0);
\draw(0,1)--(1,0);
\draw(1,4)--(4,1);
\draw(2,4)--(4,2);
\draw(3,4)--(4,3);
\end{tikzpicture}
\end{center}
\caption{Triangle $T$ (in black) and its first and second order patches (additional elements for $\UN(T)$ in dark gray and additional elements for $\UN^2(T)$ in light gray).}
\label{fig:patch}
\end{figure}

The space of $\CT_H$-piecewise affine and continuous functions is
denoted by $\CS^1(\CT_H)$.
We denote the lowest order N{\'e}d{\'e}lec finite element, cf.\ \cite[Section 5.5]{Monk}, by
\[
  \mathring{\CN}(\CT_H):=\{\Vv\in \VH_0(\curl)|\forall T\in \CT_H: \Vv|_T(\Vx)=\Va_T\times\Vx+\Vb_T \text{ with }\Va_T, \Vb_T\in\cz^3\}
\]
and the space of Raviart--Thomas fields by
\[
  \mathring{\CR\CT}(\CT_H):=\{\Vv\in \VH_0(\Div)|\forall T\in \CT_H: \Vv|_T(\Vx)=\Va_T\cdot\Vx+\Vb_T \text{ with }\Va_T\in \cz, \Vb_T\in\cz^3\}.
\]
As motivated in Section \ref{sec:motivation} we require an $\VH(\curl)$-stable interpolation operator $\pi_H^E:\VH_0(\curl)\to \mathring{\CN}(\CT_H)$ that allows for a decomposition with the estimates such as \eqref{motivation:properties-pi-H}. However, from the view point of numerical homogenization where corrector problems should be localized to small subdomains, we also need that $\pi_H^E$ is local and (as we will see later)
that it fits into a commuting diagram with other stable interpolation operators for lowest order $H^1(\Omega)$, $\VH(\Div)$ and $L^2(\Omega)$ elements.
As discussed in the introduction, the only suitable candidate is the Falk-Winther interpolation operator $\pi_H^E$ \cite{FalkWinther2014}.
We postpone a precise definition of $\pi_H^E$ to Section \ref{sec:intpolimpl} and just summarize its most important properties in the following proposition.
\begin{proposition}\label{p:proj-pi-H-E}
There exists a projection $\pi_H^E:\VH_0(\curl)\to \mathring{\CN}(\CT_H)$ with the following local stability properties:
For all $\Vv\in \VH_0(\curl)$ and all $T\in \CT_H$ it holds that
\begin{align}
\label{eq:stabilityL2}
\|\pi_H^E(\Vv)\|_{L^2(T)}&\leq C_\pi \bigl(\|\Vv\|_{L^2(\UN(T))}+H\|\curl\Vv\|_{L^2(\UN(T))}\bigr),\\*
\label{eq:stabilitycurl}
\|\curl\pi_H^E(\Vv)\|_{L^2(T)}&\leq C_\pi \|\curl\Vv\|_{L^2(\UN(T))}.
\end{align}
Furthermore, there exists a projection
$\pi_H^F:\VH_0(\Div)\to \mathring{\mathcal{RT}}(\CT_H)$
to the Raviart-Thomas space such that the following commutation property
holds
\[\curl\pi_H^E(\Vv)=\pi_H^F(\curl \Vv).\] 
\end{proposition}
\begin{proof}
 See \cite{FalkWinther2014} for a proof, which can be adapted to
 the present case of homogeneous boundary values.
\end{proof}

As explained in the motivation in Section \ref{sec:motivation}, we also require that $\pi_H^E$ allows for a regular decomposition in the sense of \eqref{motivation:properties-pi-H}. In general, regular decompositions are an important tool for the study of $\VH(\curl)$-elliptic problems and involve that a vector field $\Vv\in \VH_0(\curl)$ is split -- in a non-unique way -- into a gradient and a (regular) remainder in $\VH^1$, see \cite{Hipt02FEem, PZ02Schwarz}.  In contrast to the Helmholtz decomposition, this splitting is not orthogonal with respect to the $L^2$-inner product. If the function $\Vv\in \VH_0(\curl)$ is additionally known to be in the kernel of a suitable quasi-interpolation, a modified decomposition can be derived that is localized and $H$-weighted. In particular, the weighting with $H$ allows for estimates similar as the one stated in \eqref{motivation:properties-pi-H}. The first proof of such a modified decomposition was given by Sch\"oberl \cite{Sch08aposteriori}. In the following we shall use his results and the locality of the Falk-Winther operator to recover a similar decomposition for the projection $\pi_H^E$. More precisely, we have the following lemma which is crucial for our analysis.

\begin{lemma}
\label{lem:localregulardecomp}
Let $\pi_H^E$ denote the projection from Proposition \ref{p:proj-pi-H-E}. For any $\Vv\in \VH_0(\curl, \Omega)$, there are $\Vz\in \VH^1_0(\Omega)$ and $\theta\in H^1_0(\Omega)$ such that
\[\Vv-\pi_H^E(\Vv)=\Vz+\nabla \theta\]
with the local bounds for every $T\in \CT_H$
\begin{equation}
\label{eq:regulardecomp}
\begin{split}
H^{-1}\|\Vz\|_{L^2(T)}+\|\nabla \Vz\|_{L^2(T)}&\leq C_z\|\curl\Vv\|_{L^2(\UN^3(T))},\\
H^{-1}\|\theta\|_{L^2(T)}+\|\nabla \theta\|_{L^2(T)}&\leq C_\theta\bigl(\|\Vv\|_{L^2(\UN^3(T))}+H\|\curl\Vv\|_{L^2(\UN^3(T))}\bigr),
\end{split}
\end{equation}
where $\nabla \Vz$ stands for the Jacobi matrix of $\Vz$.
Here $C_z$ and $C_\theta$ are generic constants that only depend on the regularity of the coarse mesh. 
\end{lemma}
Observe that \eqref{eq:regulardecomp} implies the earlier formulated sufficient condition \eqref{motivation:properties-pi-H}.

\begin{proof}
Let $\Vv\in \VH_0(\curl, \Omega)$. 
Denote by $I_H^S:\VH_0(\curl,\Omega)\to \mathring{\CN}(\CT_H)$ the quasi-interpolation operator introduced by Sch\"oberl in \cite{Sch08aposteriori}.
It is shown in \cite[Theorem 6]{Sch08aposteriori} that there exists
a decomposition
\begin{equation}
\label{eq:schoeberlstab-p1}
 \Vv-I_H^S(\Vv) = 
  \sum_{\substack{P \text{ vertex}\\ \text{of }\CT_H}} \Vv_P
\end{equation}
where, for any vertex $P$,
$\Vv_P\in \VH_0(\curl, \Omega_P)$ and $\Omega_P$ the support of the local hat function associated with $P$.
Moreover, \cite[Theorem 6]{Sch08aposteriori} provides the stability
estimates
\begin{equation}\label{eq:schoeberlstab}
 \| \Vv_P \|_{L^2(\Omega_P)} \lesssim \|\Vv\|_{L^2(\UN(\Omega_P))}
\quad\text{and}\quad
 \|\curl \Vv_P \|_{L^2(\Omega_P)} 
               \lesssim \|\curl \Vv\|_{L^2(\UN(\Omega_P))}
\end{equation}
for any vertex $P$.
With these results we deduce, since $\pi_H^E$ is a projection onto the finite element space, that
\begin{align*}
\Vv-\pi_H^E(\Vv)
=\Vv-I_H^S(\Vv)-\pi_H^E(\Vv-I_H^S\Vv)
=\sum_{\substack{P \text{ vertex}\\ \text{of }\CT_H}}(\id-\pi_H^E)(\Vv_P).
\end{align*}
Due to the locality of $\pi_H^E$, we have $(\id-\pi_H^E)(\Vv_P)\in \VH_0(\curl, \UN(\Omega_P))$.
The local stability of $\pi_H^E$, \eqref{eq:stabilityL2} and \eqref{eq:stabilitycurl}, and the stability \eqref{eq:schoeberlstab} imply
\begin{align*}
\|(\id-\pi_H^E)(\Vv_P)\|_{L^2(\UN(\Omega_P))}&\lesssim \|\Vv\|_{L^2(\UN(\Omega_P))}+H\|\curl\Vv\|_{L^2(\UN(\Omega_P))},\\*
\|\curl(\id-\pi_H^E)(\Vv_P)\|_{L^2(\UN(\Omega_P))}&\lesssim \|\curl\Vv\|_{L^2(\UN(\Omega_P))},
\end{align*}
We can now apply the regular splitting to $\Vv_P$ (cf.\ \cite{PZ02Schwarz}), i.e.\ there are $\Vz_P\in \VH^1_0(\UN(\Omega_P))$, $\theta_P\in H^1_0(\UN(\Omega_P))$ such that $\Vv_P=\Vz_P+\nabla \theta_P$ and with the estimates 
\begin{align*}
H^{-1}\|\Vz_P\|_{L^2(\UN(\Omega_P))}+\|\nabla \Vz_P\|_{L^2(\UN(\Omega_P))}&\lesssim \|\curl((\id-\pi_H^E)(\Vv_P))\|_{L^2(\UN(\Omega_P))},\\*
H^{-1}\|\theta_P\|_{L^2(\UN(\Omega_P))}+\|\nabla \theta_P\|_{L^2(\UN(\Omega_P))}&\lesssim \|(\id-\pi_H^E)(\Vv_P)\|_{L^2(\UN(\Omega_P))}.
\end{align*}
Set $\Vz=\sum_P\Vz_P$ and $\theta=\sum_P\theta_P$, which is a regular decomposition of $\Vv-\pi_H^E(\Vv)$. 
The local estimates follows from the foregoing estimates for $\Vv_P$ and the decomposition \eqref{eq:schoeberlstab-p1} which yields
\begin{align*}
H^{-1}\|\Vz\|_{L^2(T)}+\|\nabla \Vz\|_{L^2(T)}&\leq \sum_{\substack{P \text{ vertex}\\ \text{of } T}} 
\left(
H^{-1}\| \Vz_P \|_{L^2(\Omega_P)}+\|\nabla \Vz_P \|_{L^2(\Omega_P)}
\right)\\
&\lesssim 
\sum_{\substack{P \text{ vertex}\\ \text{of } T}}  \|\curl (\id-\pi_H^E)(\Vv_P)\|_{L^2(\UN(\Omega_P))} 
\lesssim \|\curl\Vv\|_{L^2(\UN^3(T))}.
\end{align*}
The local estimate for $\theta$ follows analogously.
\end{proof}

\section{The Corrector Green's Operator}
\label{sec:LODideal}
In this section we introduce an ideal \emph{Corrector Green's Operator} that allows us to derive a decomposition of the exact solution into a coarse part (which is a good approximation in $H^{-1}(\Omega,\cz^3)$) and two different corrector contributions. For simplicity, we let from now on $\mathcal{L} : \VH_0(\curl) \rightarrow \VH_0(\curl)^{\prime}$ denote the differential operator associated with the sesquilinear form $\CB(\cdot,\cdot)$, i.e. $\mathcal{L}(v)(w)=\CB(v,w)$.

Using the Falk-Winter interpolation operator $\pi_H^E$ for the N{\'e}d{\'e}lec elements, we split the space $\VH_0(\curl)$ into the finite, low-dimensional coarse space $\mathring{\CN}(\CT_H)=\mbox{im}(\pi_H^E)$ and a corrector space given as the kernel of $\pi_H^E$, i.e.\ we set $\VW:=\ker (\pi_H^E)\subset \VH_0(\curl)$. This yields the direct sum splitting $\VH_0(\curl)=\mathring{\CN}(\CT_H)\oplus\VW$. Note that $\VW$ is closed since it is the kernel of a continuous (i.e. $\VH(\curl)$-stable) operator. With this the ideal Corrector Green's Operator is defined as follows.
\begin{definition}[Corrector Green's Operator]
For $\mathbf{F} \in \VH_0(\curl)^\prime$, we define the Corrector Green's Operator 
\begin{align}
\label{cor-greens-op}
\Gf: \VH_0(\curl)^{\prime} \rightarrow \VW
\hspace{40pt}
\mbox{by} \hspace{40pt}
\CB(\Gf(\mathbf{F}) , \Vw )=\mathbf{F}(\Vw)\qquad \mbox{for all } \Vw\in \VW.
\end{align}
It is well-defined by the Lax-Milgram-Babu{\v{s}}ka theorem, which is applicable since $\CB(\cdot,\cdot)$ is $\VH_0(\curl)$-elliptic and since $\VW$ is a closed subspace of $\VH_0(\curl)$.
\end{definition}
Using the Corrector Green's Operator we obtain the following decomposition of the exact solution.
\begin{lemma}[Ideal decomposition]
\label{lemma:ideal-decompos}
The exact solution $\Vu\in\VH_0(\curl)$ to \eqref{eq:problem}
and $\Vu_H:=\pi_H^E(\Vu)$ admit the decomposition
\[
\Vu = \Vu_H - (\Gf \circ \mathcal{L})(\Vu_H) + \Gf(\Vf).
\]
\end{lemma}
\begin{proof}
Since $\VH_0(\curl)=\mathring{\CN}(\CT_H)\oplus\VW$, we can write $\Vu$ uniquely as
\[
\Vu = \pi_H^E(\Vu) + (\id - \pi_H^E)(\Vu) =  \Vu_H + (\id - \pi_H^E)(\Vu),
\]
where $(\id - \pi_H^E)(\Vu) \in \VW$ by the projection property
of $\pi_H^E$.
Using the differential equation for test functions $\Vw\in \VW$ yields that
\begin{align*}
\CB(  (\id - \pi_H^E)(\Vu) , \Vw )= - \CB(  \Vu_H , \Vw ) + (\Vf, \Vw)_{L^2(\Omega)} 
= - \CB( (\Gf \circ \mathcal{L})(\Vu_H) , \Vw ) + \CB( \Gf(\Vf) , \Vw ).
\end{align*}
Since this holds for all $\Vw\in \VW$ and since $\Gf(\Vf) - (\Gf \circ \mathcal{L})(\Vu_H) \in \VW$, we conclude that
\[
(\id - \pi_H^E)(\Vu) = \Gf(\Vf) - (\Gf \circ \mathcal{L})(\Vu_H),
\]
which finishes the proof.
\end{proof}
The Corrector Green's Operator has the following approximation and stability properties, which reveal that its contributions is always negligible in the $\VH(\Div)^\prime$-norm and negligible in the $\VH(\curl)$-norm if applied to a function in $\VH(\Div)$.
\begin{lemma}[Ideal corrector estimates]
\label{lemma:corrector-props}
Any $\mathbf{F} \in \VH_0(\curl)^{\prime}$ satisfies
\begin{align}
\label{green-est-Hcurl-1}
H \| \Gf(\mathbf{F}) \|_{\VH(\curl)} + \| \Gf(\mathbf{F}) \|_{\VH(\Div)^{\prime}} \le C H \alpha^{-1} \| \mathbf{F} \|_{\VH_0(\curl)^{\prime}}.
\end{align}
If $\mathbf{F} = \mathbf{f} \in \VH(\Div)$ we even have
\begin{align}
\label{green-est-Hdiv-1}
H \| \Gf(\mathbf{f}) \|_{\VH(\curl)} + \| \Gf(\mathbf{f}) \|_{\VH(\Div)^{\prime}} \le C H^2 \alpha^{-1}  \| \mathbf{f} \|_{\VH(\Div)}.
\end{align}
Here, the constant $C$ does only depend on the maximum number of neighbors of a coarse element and the generic constants appearing in Lemma \ref{lem:localregulardecomp}.
\end{lemma}
Note that this result is still valid if we replace the $\VH(\Div)^{\prime}$-norm by the $H^{-1}(\Omega,\cz^3)$-norm.
\begin{proof}
The stability estimate $\| \Gf(\mathbf{F}) \|_{\VH(\curl)} \le \alpha^{-1} \| \mathbf{F} \|_{\VH_0(\curl)^{\prime}}$ is obvious. Next, 
with $\Gf(\mathbf{F})\in\VW$ and Lemma \ref{lem:localregulardecomp} we have
\begin{equation}\label{green-est-Hdiv-1-proof}
\begin{aligned}
\| \Gf(\mathbf{F}) \|_{\VH(\Div)^{\prime}}
&= 
\underset{\| \mathbf{v} \|_{\VH(\Div)}=1}{\sup_{\mathbf{v}\in \VH(\Div)}} \left|\int_{\Omega} \Vz \cdot \mathbf{v} - \int_{\Omega} \theta (\nabla \cdot \mathbf{v}) \right|
\\
&
\le
( \| \Vz \|_{L^2(\Omega)}^2 + \| \theta \|_{L^2(\Omega)}^2 )^{1/2}
\le C H \| \Gf(\mathbf{F}) \|_{\VH(\curl)} \le C H  \alpha^{-1} \| \mathbf{F} \|_{\VH_0(\curl)^{\prime}},
\end{aligned}
\end{equation}
which proves \eqref{green-est-Hcurl-1}.
Note that this estimate exploited $\theta \in H^{1}_0(\Omega)$, which is why we do not require the function $\mathbf{v}$ to have a vanishing normal trace. Let us now consider the case that $\mathbf{F} = \mathbf{f} \in \VH(\Div)$. We have
by \eqref{green-est-Hdiv-1-proof} that
\begin{align*}
\alpha \| \Gf( \mathbf{f} ) \|_{\VH(\curl)}^2 \le \| \Gf( \mathbf{f}) \|_{\VH(\Div)^{\prime}} 
\| \mathbf{f} \|_{\VH(\Div)} 
\le C H 
\| \Gf(\mathbf{f}) \|_{\VH(\curl)} \| \mathbf{f} \|_{\VH(\Div)}. 
\end{align*}
We conclude $\| \Gf( \mathbf{f} ) \|_{\VH(\curl)} \le C H \alpha^{-1} \| \mathbf{f} \|_{\VH(\Div)}$. Finally, we can use this estimate again in \eqref{green-est-Hdiv-1-proof} to obtain
\begin{align*}
\| \Gf(\Vf) \|_{\VH(\Div)^{\prime}} \le  C H \| \Gf(\Vf) \|_{\VH(\curl)} \le C H^2 \alpha^{-1} \| \mathbf{f} \|_{\VH(\Div)}.
\end{align*}
This finishes the proof.
\end{proof}
An immediate conclusion of Lemmas \ref{lemma:ideal-decompos} and \ref{lemma:corrector-props} is the following.

\begin{conclusion}
\label{conclusion-ideal-corr-est}
Let $\Vu$ denote the exact solution to \eqref{eq:curlcurl} for $ \mathbf{f} \in \VH(\Div)$. Then with the coarse part $\Vu_H:=\pi_H^E(\Vu)$ and corrector operator $\Kf := - \Gf \circ \mathcal{L}$ it holds
\begin{align*}
H^{-1}\| \Vu - (\id + \Kf)\Vu_H \|_{\VH(\Div)^{\prime}}
+
 \| \Vu - (\id + \Kf)\Vu_H \|_{\VH(\curl)} + \| \Vu - \Vu_H \|_{\VH(\Div)^{\prime}} \le C H \| \mathbf{f} \|_{\VH(\Div)} .
\end{align*}
Here, $C$ only depends on $\alpha$, the mesh regularity and on the constants appearing in Lemma \ref{lem:localregulardecomp}.
\end{conclusion}
\begin{proof}
 The estimates for $\Vu - (\id + \Kf)\Vu_H =\Gf(\Vf)$ directly follow
 from \eqref{green-est-Hdiv-1}.
 For the estimate of $\Vu - \Vu_H =\Kf\Vu_H + \Gf \Vf$, observe that 
\eqref{green-est-Hcurl-1} and 
 Proposition~\ref{p:proj-pi-H-E} imply
 \begin{equation*}
  \|  \Kf\Vu_H \|_{\VH(\Div)^{\prime}}
  \lesssim H
  \|  \CL\Vu_H \|_{\VH_0(\curl)^{\prime}} 
  \lesssim 
  H
  \| \Vu_H \|_{\VH(\curl)} 
  =
  H
  \| \pi_H^E \Vu \|_{\VH(\curl)} 
  \lesssim
  H
  \| \Vu \|_{\VH(\curl)} .
\end{equation*}
 Thus, the proof follows from the stability of the problem 
 and the the triangle inequality.
\end{proof}

It only remains to derive an equation that characterizes $(\id + \Kf)\Vu_H$ as the unique solution of a variational problem. This is done in the following theorem.

\begin{theorem}
We consider the setting of Conclusion \ref{conclusion-ideal-corr-est}. Then $\Vu_H=\pi_H^E(\Vu) \in \mathring{\CN}(\CT_H)$ is characterized as the unique solution to
\begin{align}
\label{ideal-lod}
\CB( \hspace{2pt} (\id + \Kf)\Vu_H , (\id + \Kf^{\ast})\Vv_H \hspace{1pt} ) = ( \Vf, (\id + \Kf^{\ast})\Vv_H )_{L^2(\Omega)} \qquad \mbox{for all } \Vv_H \in \mathring{\CN}(\CT_H).
\end{align}
Here, $\Kf^{\ast}$ is the adjoint operator to $\Kf$. The sesquilinear form $\CB( \hspace{1pt} (\id + \Kf)\hspace{3pt}\cdot \hspace{2pt}, (\id + \Kf^{\ast})\hspace{2pt}\cdot \hspace{2pt} )$ is $\VH(\curl)$-elliptic on $\mathring{\CN}(\CT_H)$.
\end{theorem}
Observe that we have the simplification $\mathcal{K}^{\ast}=\mathcal{K}$ if the differential operator $\mathcal{L}$ is self-adjoint as it is typically the case for $\VH(\curl)$-problems.
\begin{proof}
Since Lemma \ref{lemma:ideal-decompos} guarantees $\Vu = \Vu_H - (\Gf \circ \mathcal{L})(\Vu_H) + \Gf(\Vf)$, the weak formulation \eqref{eq:problem} yields
\begin{align*}
\CB( \Vu_H - (\Gf \circ \mathcal{L})(\Vu_H) + \Gf(\Vf) , \Vv_H ) = ( \Vf, \Vv_H )_{L^2(\Omega)} \qquad \mbox{for all } \Vv_H \in \mathring{\CN}(\CT_H).
\end{align*}
We observe that by definition of $\Gf$ we have
\begin{align*}
\CB( \Gf(\Vf) , \Vv_H ) = ( \Vf , (\Gf \circ \mathcal{L})^{\ast}\Vv_H )_{L^2(\Omega)}
\end{align*}
and
\begin{align*}
\CB( \Vu_H - (\Gf \circ \mathcal{L})(\Vu_H)  , (\Gf \circ \mathcal{L})^{\ast}\Vv_H ) = 0.
\end{align*}
Combining the three equations shows that $(\id + \Kf)\Vu_H$ is a solution to \eqref{ideal-lod}. The uniqueness follows from the
following norm equivalence
\begin{align*}
\| \Vu_H \|_{\VH(\curl)} = \| \pi_H^E((\id + \Kf)\Vu_H) \|_{\VH(\curl)} \le C \| (\id + \Kf)\Vu_H \|_{\VH(\curl)} 
\le C \| \Vu_H \|_{\VH(\curl)}.
\end{align*}
This is also the reason why the $\VH(\curl)$-ellipticity of $\CB( \cdot, \cdot)$ implies the $\VH(\curl)$-ellipticity of $\CB( \hspace{1pt} (\id + \Kf)\hspace{3pt}\cdot \hspace{2pt}, (\id + \Kf^{\ast})\hspace{2pt}\cdot \hspace{2pt} )$ on $\mathring{\CN}(\CT_H)$.
\end{proof}
\textbf{Numerical homogenization}. Let us summarize the most important findings and relate them to (numerical) homogenization. We defined a \emph{homogenization scale} through the coarse FE space $\mathring{\CN}(\CT_H)$. We proved that 
there exists a numerically homogenized function $\Vu_H \in \mathring{\CN}(\CT_H)$ which approximates the exact solution well in $\VH(\Div)^{\prime}$ with
\begin{align*}
 \| \Vu - \Vu_H \|_{\VH(\Div)^{\prime}} \le C H \| \mathbf{f} \|_{\VH(\Div)}.
\end{align*}
From the periodic homogenization theory (cf. Section \ref{sec:motivation}) we know that this is the best we can expect and that $\Vu_H$ is typically not a good $L^2$-approximation due to the large kernel of the curl-operator.
Furthermore, we showed the existence of an $\VH(\curl)$-stable corrector operator $\Kf: \mathring{\CN}(\CT_H) \rightarrow \VW$ that corrects the homogenized solution in such a way that the approximation is also accurate in $\VH(\curl)$ with
\begin{align*}
 \| \Vu - (\id + \Kf)\Vu_H \|_{\VH(\curl)} \le C H \| \mathbf{f} \|_{\VH(\Div)}.
\end{align*}
Since $\Kf = - \Gf \circ \mathcal{L}$, we know that we can characterize $\Kf (\Vv_H) \in \VW$ as the unique solution to the (ideal) corrector problem 
\begin{align}
\label{ideal-corrector-problem}
\CB( \Kf (\Vv_H) , \Vw )=- \CB( \Vv_H , \Vw ) \qquad \mbox{for all } \Vw\in \VW.
\end{align}
The above result shows that $(\id + \Kf)\Vu_H$ approximates the analytical solution with linear rate without any assumptions on the regularity of the problem or the structure of the coefficients that define $\CB(\cdot,\cdot)$. Also it does not assume that the mesh resolves the possible fine-scale features of the coefficient.
On the other hand, the ideal corrector problem \eqref{ideal-corrector-problem} is global, which significantly limits its practical usability in terms of real computations.

However, as we will see next, the corrector Green's function associated with problem \eqref{cor-greens-op} shows an exponential decay measured in units of $H$. This property will allow us to split the global corrector problem \eqref{ideal-corrector-problem} into several smaller problems on subdomains, similar to how we encounter it in classical homogenization theory. We show the exponential decay of the corrector Green's function indirectly through the properties of its corresponding Green's operator $\Gf$. The localization is established in Section \ref{sec:LOD}, whereas we prove the decay in Section \ref{sec:decaycorrectors}.

\section{Quasi-local numerical homogenization}
\label{sec:LOD}
In this section we describe how the ideal corrector $\Kf$ can be approximated by a sum of local correctors, without destroying the overall approximation order. This is of central importance for an efficient computability. Furthermore, it also reveals that the new corrector is a quasi-local operator, which is in line with homogenization theory. 

We start with quantifying the decay properties of the Corrector Green's Operator in Section \ref{subsec:idealapprox}. In Section \ref{subsec:LODlocal} we apply the result to our numerical homogenization setting and state the error estimates for the \quotes{localized} corrector operator. We close with a few remarks on a fully discrete realization of the localized corrector operator in Section \ref{subsec:discreteLOD}.
\subsection{Exponential decay and localized corrector}
\label{subsec:idealapprox}
The property that $\Kf$ can be approximated by local correctors is directly linked to the decay properties of the Green's function associated with problem \eqref{cor-greens-op}. These decay properties can be quantified explicitly by measuring distances between points in units of the coarse mesh size $H$. We have the following result, which states -- loosely speaking -- in which distance from the support of a source term $\mathbf{F}$, becomes the $\VH(\curl)$-norm of $\Gf(\mathbf{F})$ negligibly small. For that, recall the definition of the element patches from the beginning of Section \ref{sec:intpol}, where $\UN^m(T)$ denotes the patch that consists of a coarse element $T \in  \CT_H$ and $m$ layers of coarse elements around it. A proof of the following proposition is given in Section \ref{sec:decaycorrectors}. 
\begin{proposition}
\label{prop:decaycorrector1}
Let $T\in \CT_H$ denote a coarse element and $m\in \nz$ a number of layers. Furthermore, let $\mathbf{F}_T \in \VH_0(\curl)^{\prime}$ denote a local source functional in the sense that $\mathbf{F}_T(\Vv)=0$ for all $\Vv \in \VH_0(\curl)$ with $\supp(\Vv) \subset \Omega \setminus T$. Then there exists $0<\tilde{\beta}<1$, independent of $H$, $T$, $m$ and $\mathbf{F}_T$, such that 
\begin{equation}
\label{eq:decaycorrector1}
\| \Gf(\mathbf{F}_T) \|_{\VH(\curl, \Omega\setminus \UN^m(T))}\lesssim \tilde{\beta}^m\| \mathbf{F}_T \|_{\VH_0(\curl)^{\prime}}.
\end{equation}
\end{proposition}

In order to use this result to approximate $\Kf(\Vv_H) = - (\Gf \circ \mathcal{L})\Vv_H$ (which has a nonlocal argument), we introduce,
for any $T\in\CT_H$, localized differential operators 
$\CL_T:\VH(\curl,T)\to\VH(\curl,\Omega)'$
with
\[\langle \mathcal{L}_T(\Vu), \Vv \rangle := \CB_T(\Vu, \Vv ),\]
where $\CB_T(\cdot, \cdot )$ denotes the restriction of $\CB(\cdot, \cdot )$ to the element $T$. By linearity of $\Gf$ we have that 
\[\Gf \circ \mathcal{L} = \sum_{T \in \CT_H} \Gf \circ \mathcal{L}_T\]
and consequently we can write
\[
\Kf( \Vv_H ) =  \sum_{T \in \CT_H} \Gf( \mathbf{F}_T ), \qquad \mbox{with } \mathbf{F}_T:= - \mathcal{L}_T(\Vv_H).
\]
Obviously, $\Gf( \mathbf{F}_T )$ fits into the setting of Proposition \ref{prop:decaycorrector1}. This suggests to truncate the individual computations of $\Gf( \mathbf{F}_T )$ to a small patch $\UN^m(T)$ and then collect the results to construct a global approximation for the corrector. Typically, $m$ is referred to as \emph{oversampling parameter}. The strategy is detailed in the following definition.
\begin{definition}[Localized Corrector Approximation]
\label{de:loc-correctors}
For an element $T\in \CT_H$ we define the element patch $\Omega_T:=\UN^m(T)$ of order $m\in \nz$. Let 
$\mathbf{F} \in \VH_0(\curl)^{\prime}$ be the sum of local functionals with $\mathbf{F} =\sum_{T\in \CT_H} \mathbf{F}_T$, where $\mathbf{F}_T \in \VH_0(\curl)^{\prime}$ is as in Proposition \ref{prop:decaycorrector1}. Furthermore, let $\VW(\Omega_T)\subset \VW$ denote the space of functions from $\VW$ that vanish outside $\Omega_T$, i.e.
\[\VW(\Omega_T)=\{\Vw\in\VW|\Vw=0 \text{ \textrm{outside} }\Omega_T\}.\]
We call $\Gf_{T,m}( \mathbf{F}_T ) \in \VW(\Omega_T)$ the \emph{localized corrector} if it solves
\begin{equation}
\label{eq:correctorlocal}
\CB( \Gf_{T,m}( \mathbf{F}_T ) , \Vw )=\mathbf{F}_T(\Vw)\qquad \mbox{for all } \Vw\in \VW(\Omega_T).
\end{equation}
With this, the global corrector approximation is given by
\begin{align*}
\Gf_{m}(\mathbf{F}) := \sum_{T\in \CT_H} \Gf_{T,m}( \mathbf{F}_T ).
\end{align*}
\end{definition}
Observe that problem \eqref{eq:correctorlocal} is only formulated on the patch $\Omega_T$ and that it admits a unique solution by the Lax-Milgram-Babu{\v{s}}ka theorem.

Based on decay properties stated in Proposition \ref{prop:decaycorrector1}, we can derive the following error estimate for the difference between the exact corrector $\Gf(\mathbf{F})$ and its approximation $\Gf_{m}(\mathbf{F})$ obtained by an $m$th level truncation. The proof of the following result is again postponed to Section \ref{sec:decaycorrectors}.
\begin{theorem}
\label{thm:errorcorrectors}
We consider the setting of Definition \ref{de:loc-correctors} with ideal Green's Corrector $\Gf(\mathbf{F})$ and its $m$th level truncated approximation $\Gf_{m}(\mathbf{F})$. Then there exist constants $C_{d}>0$ and $0<\beta<1$ (both independent of $H$ and $m$) such that
\begin{align}
\label{eq:errorcorrector}
\| \Gf(\mathbf{F}) - \Gf_{m}(\mathbf{F}) \|_{\VH(\curl)}&\leq C_{d} \sqrt{C_{\ol,m}}\,\beta^m \left( \sum_{T\in \CT_H} \| \mathbf{F}_T \|_{\VH_0(\curl)^{\prime}}^2 \right)^{1/2}
\end{align}
and
\begin{align}
\label{eq:errorcorrector-2}
\| \Gf(\mathbf{F}) - \Gf_{m}(\mathbf{F}) \|_{\VH(\Div)^{\prime}}&\leq C_{d} \sqrt{C_{\ol,m}}\, \beta^m H \left( \sum_{T\in \CT_H} \| \mathbf{F}_T \|_{\VH_0(\curl)^{\prime}}^2 \right)^{1/2}.
\end{align}
\end{theorem}
As a direct conclusion from Theorem \ref{thm:errorcorrectors} we obtain the main result of this paper that we present in the next subsection.

\subsection{The quasi-local corrector and homogenization} 
\label{subsec:LODlocal} 
Following the above motivation we split the ideal corrector $\Kf(\Vv_H) =- (\Gf \circ \mathcal{L})\Vv_H$ into a sum of quasi-local contributions of the form $\sum_{T \in \CT_H} (\Gf \circ \mathcal{L}_T)\Vv_H$. Applying Theorem \ref{thm:errorcorrectors}, we obtain the following result.
\begin{conclusion}
\label{conclusion-main-result}
Let $\Kf_m := -  \sum_{T \in \CT_H} (\Gf_{T,m} \circ \mathcal{L}_T): \mathring{\CN}(\CT_H) \rightarrow \VW$ denote the localized corrector operator obtained by truncation of $m$th order. Then it holds 
\begin{align}
\label{conclusion-main-result-est}
\inf_{\mathbf{v}_H \in \mathring{\CN}(\CT_H)} \| \Vu - (\id + \Kf_m)\mathbf{v}_H \|_{\VH(\curl)} \le 
C \left( H  + \sqrt{C_{\ol,m}} \beta^m \right) \| \Vf \|_{\VH(\Div)}.
\end{align}
\end{conclusion}
Note that even though the ideal corrector $\Kf$ is a non-local operator, we can approximate it by a quasi-local corrector $\Kf_m$. Here, the quasi-locality is seen by the fact that, if $\Kf$ is applied to a function $\Vv_H$ with local support, the image $\Kf(\Vv_H)$ will typically still have a global support in $\Omega$. On the other hand, if $\Kf_m$ is applied to a locally supported function, the support will only increase by a layer with thickness of order $mH$.
\begin{proof}[Proof of Conclusion \ref{conclusion-main-result}]
With $\Kf_m = -  \sum_{T \in \CT_H} (\Gf_{T,m} \circ \mathcal{L}_T)$ we apply
Conclusion~\ref{conclusion-ideal-corr-est}
and Theorem \ref{thm:errorcorrectors} to obtain
\begin{equation*}
\begin{aligned}
&
\inf_{\mathbf{v}_H \in \mathring{\CN}(\CT_H)} \| \Vu - (\id + \Kf_m)\mathbf{v}_H \|_{\VH(\curl)}
\le
\| \Vu - (\id + \Kf)\mathbf{u}_H \|_{\VH(\curl)} + \|(\Kf - \Kf_m)\mathbf{u}_H \|_{\VH(\curl)}\\
&
\qquad\qquad\qquad\qquad
\le C H \| \Vf \|_{\VH(\Div)} + C \sqrt{C_{\ol,m}}\, \beta^m \left( \sum_{T\in \CT_H} \| \mathcal{L}_T(\mathbf{u}_H) \|_{\VH_0(\curl)^{\prime}}^2 \right)^{1/2},
\end{aligned}
\end{equation*}
where we observe with $\| \mathcal{L}_T(\mathbf{v}_H) \|_{\VH_0(\curl)^{\prime}} \le C \| \mathbf{v}_H \|_{\VH(\curl,T)}$ that
\begin{align*}
\sum_{T\in \CT_H} \| \mathcal{L}_T(\mathbf{u}_H) \|_{\VH_0(\curl)^{\prime}}^2 \le C \| \mathbf{u}_H \|_{\VH(\curl)}^2
= C \| \pi_H^E(\Vu) \|_{\VH(\curl)}^2 \le C \| \Vu \|_{\VH(\curl)}^2 \le C \| \Vf \|_{\VH(\Div)}^2.
\end{align*}
\end{proof}
Conclusion \ref{conclusion-main-result} has immediate implications from the computational point of view. First, we observe that $\Kf_m$ can be computed by solving local decoupled problems. Considering a basis $\{ \boldsymbol{\Phi}_k | \hspace{3pt} 1 \le k \le N \}$ of $\mathring{\CN}(\CT_H)$, we require to determine $\Kf_m(\boldsymbol{\Phi}_k)$. For that, we consider all $T \in \CT_H$ with $T \subset \supp(\boldsymbol{\Phi}_k)$ and solve for $\Kf_{T,m}(\boldsymbol{\Phi}_k) \in \VW(\hspace{1pt}\UN^m(T)\hspace{1pt})$ with
\begin{align}
\label{loc-corrector-problems}
\CB_{\UN^m(T)}( \Kf_{T,m}(\boldsymbol{\Phi}_k), \Vw ) = - \CB_{T}( \boldsymbol{\Phi}_k , \Vw ) \qquad \mbox{for all } 
\Vw \in \VW(\hspace{1pt}\UN^m(T)\hspace{1pt}).
\end{align}
The global corrector approximation is now given by
\[
\Kf_m(\boldsymbol{\Phi}_k) = \underset{ T \subset \supp(\boldsymbol{\Phi}_k) }{\sum_{ T \in \CT_H }}
 \Kf_{T,m}(\boldsymbol{\Phi}_k).
\]
Next, we observe that selecting the localization parameter $m$ such that 
\[ 
    m\gtrsim \lvert \log H\rvert \big/ \lvert \log \beta\rvert,
\]
we have with Conclusion \ref{conclusion-main-result} that
\begin{align}
\label{curl-est-m-logH}\inf_{\mathbf{v}_H \in \mathring{\CN}(\CT_H)} \| \Vu - (\id + \Kf_m)\mathbf{v}_H \|_{\VH(\curl)} \le 
C H \| \Vf \|_{\VH(\Div)},
\end{align}
which is of the same order as for the ideal corrector $\mathcal{K}$. Consequently, we can consider the Galerkin finite element method, where we seek $\Vu_{H,m} \in \mathring{\CN}(\CT_H)$ such that
\begin{align*}
\CB( (\id + \Kf_m)\Vu_{H,m} , (\id + \Kf_m)\mathbf{v}_H ) = (\mathbf{f} , (\id + \Kf_m)\mathbf{v}_H )_{L^2(\Omega)} 
\qquad \mbox{for all } \Vv_{H,m} \in \mathring{\CN}(\CT_H).
\end{align*}
Since a Galerkin method yields the $\VH(\curl)$-quasi-best approximation of $\Vu$ in the space \linebreak[4]$(\id + \Kf_m)\mathring{\CN}(\CT_H)$ we have with \eqref{curl-est-m-logH} that
\begin{align*}
\| \Vu - (\id + \Kf_m)\Vu_{H,m} \|_{\VH(\curl)} \le C H \| \Vf \|_{\VH(\Div)}
\end{align*}
and we have with \eqref{green-est-Hcurl-1}, \eqref{eq:errorcorrector-2} and the $\VH(\curl)$-stability of $\pi_H^E$ that
\begin{align*}
\| \Vu - \Vu_{H,m} \|_{\VH(\Div)^{\prime}} \le C H \| \Vf \|_{\VH(\Div)}.
\end{align*}
This result is a homogenization result in the sense that it yields a coarse function $\Vu_{H,m}$ that approximates the exact solution in $\VH(\Div)^{\prime}$. Furthermore, it yields an appropriate (quasi-local) corrector $\Kf_m(\Vu_{H,m})$ that is required for an accurate approximation in $\VH(\curl)$.
\begin{remark}[Refined estimates]
With a more careful proof, the constants in the estimate of Conclusion \ref{conclusion-main-result} can be specified as  
\begin{eqnarray*}
\label{eq:errorLOD-refined}
\lefteqn{\inf_{\mathbf{v}_H \in \mathring{\CN}(\CT_H)} \| \Vu - (\id + \Kf_m)\mathbf{v}_H \|_{\VH(\curl)}}\\ 
\nonumber&\leq& \alpha^{-1}(1+H)\bigl(H\max\{C_z, C_\theta\} \sqrt{C_{\ol,3}}+C_d C_\pi C_B^2\sqrt{C_{\ol,m}C_{\ol}}\, \beta^m\bigr)\|\Vf\|_{\VH(\Div)},
\end{eqnarray*}
where $\alpha$ and $C_B$ are as in Assumption \ref{asspt:sesquiform}, $C_{d}$ is the constant appearing in the decay estimate \eqref{eq:errorcorrector}, $C_\pi$ is as in Proposition \ref{p:proj-pi-H-E}, $C_z$ and $C_\theta$ are from \eqref{eq:regulardecomp} and $C_{\ol,m}$ as detailed at the beginning of Section \ref{sec:intpol}.
Note that if $m$ is large enough so that $\UN^m(T)=\Omega$ for all $T \in \CT_H$, we have as a refinement of Conclusion \ref{conclusion-ideal-corr-est} that
\begin{eqnarray*}
\inf_{\mathbf{v}_H \in \mathring{\CN}(\CT_H)} \| \Vu - (\id + \Kf)\mathbf{v}_H \|_{\VH(\curl)} \leq \alpha^{-1}(1+H)\bigl(H\max\{C_z, C_\theta\} \sqrt{C_{\ol,3}} \bigr)\|\Vf\|_{\VH(\Div)}.
\end{eqnarray*}
\end{remark}
 
\subsection{A fully discrete localized multiscale method}
\label{subsec:discreteLOD}
The procedure described in the previous section is still not yet \quotes{ready to use} for a practical computation as the local corrector problems \eqref{loc-corrector-problems} involve the infinite dimensional spaces $\VW(\Omega_T)$. Hence, we require an additional fine scale discretization of the corrector problems (just like the cell problems in periodic homogenization theory can typically not be solved analytically). 

For a fully discrete formulation, we introduce a second shape-regular partition $\CT_h$ of $\Omega$ into tetrahedra.
This partition may be non-uniform and is assumed to be obtained from $\CT_H$ by at least one global refinement.
It is a fine discretization in the sense that $h<H$ and that $\CT_h$ resolves all fine-scale features of the coefficients.
Let $\mathring{\CN}(\CT_h)\subset\VH_0(\curl)$ denote the space of N{\'e}d{\'e}lec elements with respect to the partition $\CT_h$.
We then introduce the space 
\[\VW_h(\Omega_T):=\VW(\Omega_T)\cap\mathring{\CN}(\CT_h)=\{\Vv_h\in\mathring{\CN}(\CT_h)|\Vv_h=0\text{ outside }\Omega_T, \pi_H^E(\Vv_h)=0\}\]
and discretize the corrector problem 
\eqref{loc-corrector-problems} with this new space.
The corresponding correctors are denoted by $\Kf_{T,m,h}$ and $\Kf_{m,h}$. With this modification we can prove analogously to the error estimate \eqref{conclusion-main-result-est} that it holds
\begin{align}
\label{conclusion-main-result-est-h}
\inf_{\mathbf{v}_H \in \mathring{\CN}(\CT_H)} \| \Vu_h - (\id + \Kf_{m,h})\mathbf{v}_H \|_{\VH(\curl)} \le 
C \left( H  + \sqrt{C_{\ol,m}} \tilde{\beta}^m \right) \| \Vf \|_{\VH(\Div)},
\end{align}
where $\Vu_h$ is the Galerkin approximation of $\Vu$ in the discrete fine space $\mathring{\CN}(\CT_h)$. If $\CT_h$ is fine enough, we can assume that $\Vu_h$ is a good $\VH(\curl)$-approximation to the true solution $\Vu$. Consequently, it is justified to formulate a fully discrete (localized) multiscale method by seeking
$\Vu_{H,h,m}^{\ms}:=(\id+\Kf_{m,h})\Vu_H$ with $\Vu_H\in \mathring{\CN}(\CT_H)$ such that
\begin{equation}
\label{eq:discreteLOD}
\CB(\Vu_{H,h,m}^{\ms}, (\id+\Kf_{m,h})\Vv_H)=(\Vf, (\id+\Kf_{m,h})\Vv_H)_{L^2(\Omega)}\qquad\mbox{for all } \Vv_H\in\mathring{\CN}(\CT_H).
\end{equation}
As before, we can conclude from \eqref{conclusion-main-result-est-h} together with the choice $m\gtrsim \lvert \log H\rvert/\lvert\log \beta\rvert$, that it holds
\begin{align*}
\| \Vu_h - \Vu_{H,h,m}^{\ms} \|_{\VH(\curl)}
+
\| \Vu_h - \pi_H^E \Vu_{H,h,m}^{\ms} \|_{\VH(\Div)^{\prime}} 
\le C H \| \Vf \|_{\VH(\Div)}.
\end{align*}
Thus, the additional fine-scale discretization does not affect the overall error estimates and we therefore concentrate in the proofs (for simplicity) on the semi-discrete case as detailed in Sections \ref{subsec:idealapprox} and \ref{subsec:LODlocal}. Compared to the fully-discrete case, only some small modifications are needed in the proofs for the decay of the correctors. These modifications are outlined at the end of Section \ref{sec:decaycorrectors}. Note that $\Vu_h$ is not needed in the practical implementation of the method.

\section{Proof of the decay for the Corrector Green's Operator}
\label{sec:decaycorrectors}
In this section, we prove Proposition \ref{prop:decaycorrector1} and Theorem \ref{thm:errorcorrectors}. Since the latter one is based on the first result, we start with proving
the exponential decay of the Green's function associated with $\Gf$. Recall that we quantified the decay indirectly through estimates of the form
\begin{equation*}
\| \Gf(\mathbf{F}_T) \|_{\VH(\curl, \Omega\setminus \UN^m(T))}\lesssim \tilde{\beta}^m\| \mathbf{F}_T \|_{\VH_0(\curl)^{\prime}},
\end{equation*}
where $\mathbf{F}_T$ is a $T$-local functional and $0<\tilde{\beta}<1$.
\begin{proof}[Proof of Proposition \ref{prop:decaycorrector1}]
Let $\eta\in \CS^1(\CT_H)\subset H^1(\Omega)$ be a scalar-valued, piece-wise linear and globally continuous cut-off function with
\begin{equation*}
\eta=0\qquad \text{in}\quad \UN^{m-6}(T)\qquad \qquad\qquad \eta=1\qquad \text{in}\quad \Omega\setminus\UN^{m-5}(T).
\end{equation*}
Denote $\CR=\supp(\nabla \eta)$ and $\Vphi:=\Gf(\mathbf{F}_T) \in \VW$. In the following we use $\UN^k(\CR)=\UN^{m-5+k}(T)\setminus \UN^{m-6-k}(T)$.
Note that $\|\nabla \eta\|_{L^\infty(\CR)}\sim H^{-1}$.
Furthermore, let $\Vphi=\Vphi-\pi_H^E\Vphi=\Vz+\nabla \theta$ be the splitting from Lemma \ref{lem:localregulardecomp}.
We obtain with $\eta\leq 1$, the coercivity, and the product rule
that
\begin{align*}
\alpha\|\Vphi\|^2_{\VH(\curl, \Omega\setminus\UN^m(T))}&\leq \bigl|(\mu\curl\Vphi, \eta\curl\Vphi)_{L^2(\Omega)}+(\kappa\Vphi, \eta\Vphi)_{L^2(\Omega)}\bigr|\\
&=\bigl|(\mu\curl\Vphi, \eta\curl\Vz)_{L^2(\Omega)}+(\kappa\Vphi, \eta\nabla\theta+\eta\Vz)_{L^2(\Omega)}\bigr|\\
&\leq M_1 
+M_2+M_3+M_4+M_5,
\end{align*}
where 
\begin{align*}
& M_1:=\Bigl|\bigl(\mu\curl\Vphi, \curl(\id-\pi_H^E)(\eta\Vz)\bigr)_{L^2(\Omega)}
&&\hspace{-23pt}+\enspace
\bigl(\kappa \Vphi, (\id-\pi_H^E)
(\eta\Vz+\nabla(\eta\theta))\bigr)_{L^2(\Omega)}\Bigr|,
\\
&M_2:=\Bigl|\bigl(\mu\curl\Vphi, \curl\pi_H^ E(\eta\Vz)\bigr)_{L^2(\Omega)}\Bigr|,
&&
M_3:=\Bigl|\bigl(\kappa \Vphi,\pi_H^E(\eta\Vz+\nabla(\eta\Vphi))\bigr)_{L^2(\Omega)}\Bigr|,
\\
&
M_4:=\Bigl|\bigl(\mu\curl \Vphi, \nabla \eta\times \Vz\bigr)_{L^2(\Omega)}\Bigr|,
&&
M_5:=\Bigl|\bigl(\kappa\Vphi, \theta\nabla \eta\bigr)_{L^2(\Omega)}\Bigr|. 
\end{align*}
We used the product rule $\curl(\eta\Vz)=\nabla\eta\times \Vz+\eta\curl\Vz$ here.

We now estimate the five terms separately.
Let $\Vw:=(\id-\pi_H^E)(\eta\Vz+\nabla(\eta\theta))$ and note that  (i) $\curl\Vw=\curl(\id-\pi_H^E)(\eta\Vz)$, (ii) $\Vw\in \VW$, (iii) $\supp\Vw\subset\Omega\setminus T$.
Using the definition of the Corrector Green's Operator in \eqref{cor-greens-op} and the fact that $\mathbf{F}_T(\Vw)=0$ yields $M_1=0$.

For $M_2$, note that the commuting property of the projections
$\pi^E$ and $\pi^F$ implies
$\curl\pi_H^E(\Vz)=\pi_H^F(\curl \Vz)=\pi_H^F(\curl\Vphi)=\curl\pi_H^E\Vphi=0$
because $\Vphi\in \VW$.
Using the stability of $\pi_H^E$ \eqref{eq:stabilitycurl} and Lemma \ref{lem:localregulardecomp}, we can estimate $M_2$ as
\begin{align*}
M_2&\lesssim \|\curl\Vphi\|_{L^2(\UN(\CR))}\|\curl\pi_H^E(\eta\Vz)\|_{L^2(\UN(\CR))}\lesssim \|\curl\Vphi\|_{L^2(\UN(\CR))}\|\curl(\eta\Vz)\|_{L^2(\UN^2(\CR))}\\
&\lesssim \|\curl\Vphi\|_{L^2(\UN(\CR))}\Bigl(\|\nabla\eta\|_{L^\infty(\CR)}\|\Vz\|_{L^2(\CR)}
 +\|\eta\|_{L^\infty(\UN^2(\CR))}\|\curl\Vz\|_{L^2(\UN^{m-3}(T)\setminus \UN^{m-6}(T)))}\Bigr)\\
&\lesssim \|\curl\Vphi\|_{L^2(\UN(\CR))}\|\curl\Vphi\|_{L^2(\UN^{m}(T)\setminus \UN^{m-9}(T))}.
\end{align*}

In a similar manner, we obtain for $M_3$ that
\begin{align*}
M_3&\lesssim\|\Vphi\|_{L^2(\UN(\CR))}\Bigl(\|\eta \Vz\|_{L^2(\UN^2(\CR))}+\|\nabla(\eta\theta)\|_{L^2(\UN^2(\CR))}+H\|\curl(\eta\Vz)\|_{L^2(\UN^2(\CR))}\Bigr)\\
&\lesssim \|\Vphi\|_{L^2(\UN(\CR))}\Bigl(\|\Vphi\|_{L^2(\UN^{m}(T)\setminus \UN^{m-9}(T))}+H\|\curl\Vphi\|_{L^2(\UN^{m}(T)\setminus \UN^{m-9}(T))}\Bigr).
\end{align*}

Simply using Lemma \ref{lem:localregulardecomp}, we deduce for $M_4$ and $M_5$
\begin{align*}
M_4&\lesssim \|\curl\Vphi\|_{L^2(\CR)}\|\curl\Vphi\|_{L^2(\UN^3(\CR))},
\\
M_5&\lesssim \|\Vphi\|_{L^2(\CR)}
    (\|\Vphi\|_{L^2(\UN^3(\CR))}
      + H\|\curl \Vphi\|_{L^2(\UN^3(\CR))}).
\end{align*}
All in all, this gives
\begin{equation*}
\|\Vphi\|^2_{\VH(\curl, \Omega\setminus \UN^m(T))}\leq 
\tilde{C} \|\Vphi\|^2_{\VH(\curl, \UN^{m}(T)\setminus \UN^{m-9}(T) )}
\end{equation*}
for some $\tilde{C}>0$. 
Moreover, it holds that
\begin{equation*}
\|\Vphi\|^2_{\VH(\curl, \Omega\setminus \UN^m(T))}
=
\|\Vphi\|^2_{\VH(\curl, \Omega\setminus \UN^{m-9}(T))}
- \|\Vphi\|^2_{\VH(\curl, \UN^m(T)\setminus \UN^{m-9}(T))}.
\end{equation*}
Thus, we obtain finally with $\tilde{\beta}_{\pre}:=(1+\tilde{C}^{-1})^{-1}<1$, a re-iteration of the above argument, and Lemma~\ref{lemma:corrector-props} that
\begin{equation*}
\|\Vphi\|^2_{\VH(\curl, \Omega\setminus \UN^m(T))}\lesssim \tilde{\beta}_{\pre}^{\lfloor m/9\rfloor}\|\Vphi\|^2_{\VH(\curl)}\lesssim \tilde{\beta}_{\pre}^{\lfloor m/9\rfloor}\| \mathbf{F}_T \|^2_{\VH_0(\curl)^\prime}.
\end{equation*}
Algebraic manipulations give the assertion.
\end{proof}
\begin{proof}[Proof of Theorem \ref{thm:errorcorrectors}]
We start by proving the following local estimate
\begin{align}
\label{eq:errorcorrectorlocal}
\| \Gf( \mathbf{F}_T )-\Gf_{T,m}( \mathbf{F}_T ) \|_{\VH(\curl)}&\leq C_1 \tilde{\beta}^m \| \mathbf{F}_T \|_{\VH_0(\curl)^{\prime}}
\end{align}
for some constant $C_1>0$ and $0<\tilde{\beta}<1$.
Let $\eta\in \CS^1(\CT_H)$ be a piece-wise linear and globally continuous cut-off function with
\begin{align*}
\eta=0 \qquad \text{in} \quad \Omega\setminus \UN^{m-1}(T)\qquad\qquad\qquad\eta=1\qquad\text{in}\quad \UN^{m-2}(T).
\end{align*}
Due to C{\'e}a's Lemma we have
\begin{align*}
\| \Gf( \mathbf{F}_T ) - \Gf_{T,m}( \mathbf{F}_T ) \|_{\VH(\curl)}\lesssim \inf_{\Vw_{T,m}\in \VW(\Omega_T)}\| \Gf( \mathbf{F}_T ) -\Vw_{T,m}\|_{\VH(\curl)}.
\end{align*}
We use the splitting of Lemma \ref{lem:localregulardecomp} and write $\Gf( \mathbf{F}_T )=(\id-\pi_H^E)(\Gf( \mathbf{F}_T ))=\Vz+\nabla\theta$.
Then we choose $\Vw_{T,m}=(\id-\pi_H^E)(\eta\Vz+\nabla(\eta\theta))\in \VW(\Omega_T)$ and derive
with the stability of $\pi_H^E$ and \eqref{eq:regulardecomp}
\begin{align*}
\|\Gf( \mathbf{F}_T )-\Gf_{T,m}( \mathbf{F}_T )\|_{\VH(\curl)}&\lesssim \|(\id-\pi_H^E)(\Gf( \mathbf{F}_T )-\eta\Vz - \nabla(\eta\theta))\|_{\VH(\curl)}\\
&=\|(\id-\pi_H^E)((1-\eta)\Vz+\nabla((1-\eta)\theta))\|_{\VH(\curl)}\\
&\lesssim \|(1-\eta)\Vz\|_{L^2(\Omega\setminus\{\eta=1\})}+\|\nabla((1-\eta)\theta)\|_{L^2(\Omega\setminus\{\eta=1\})}\\*
&\quad+(1+H)\|\curl((1-\eta)\Vz)\|_{L^2(\Omega\setminus\{\eta=1\})}\\
&\lesssim (1+H)\,\| \Gf( \mathbf{F}_T )\|_{\VH(\curl, \UN^3(\Omega\setminus\{\eta=1\}))}.
\end{align*}
Combination with Proposition \ref{prop:decaycorrector1} gives estimate \eqref{eq:errorcorrectorlocal}.

To prove the main estimate of Theorem \ref{thm:errorcorrectors}, i.e.\ estimate \eqref{eq:errorcorrector}, we define, for a given simplex $T\in\CT_H$, the piece-wise linear, globally continuous cut-off function $\eta_T\in \CS^1(\CT_H)$ via
\begin{align*}
\eta_T=0\qquad \text{in}\quad   \UN^{m+1}(T)\qquad\qquad\qquad \eta_T=1\qquad\text{in}\quad \Omega\setminus\UN^{m+2}(T).
\end{align*}
Denote $\Vw:=(\Gf-\Gf_m)(\mathbf{F})=\sum_{T \in \CT_H} \Vw_T$ with $\Vw_T:=(\Gf-\Gf_{T,m})(\mathbf{F}_T)$ and split $\Vw$ according to Lemma \ref{lem:localregulardecomp} as $\Vw=\Vw-\pi_H^E(\Vw)=\Vz+\nabla\theta$.
Due to the ellipticity of $\CB$ and its sesquilinearity, we have
\begin{align*}
\alpha\|\Vw\|^2_{\VH(\curl)}
\leq
 \Bigl|\sum_{T\in\CT_H}\CB(\Vw_T,\Vw)\Bigr|\leq \sum_{T\in \CT_H}|\CB(\Vw_T,\Vz+\nabla\theta )|
\leq
\sum_{T\in \CT_H} (A_T + B_T)
\end{align*}
where, for any $T\in\CT_H$, we abbreviate
\begin{equation*}
 A_T:=|\CB(\Vw_T,(1-\eta_T)\Vz+\nabla((1-\eta_T)\theta))|
 \quad\text{and}\quad
 B_T:=|\CB(\Vw_T,\eta_T\Vz+\nabla(\eta_T\theta))| .
\end{equation*}

For the term $A_T$, we derive by using the properties of the cut-off function and the regular decomposition \eqref{eq:regulardecomp}
\begin{align*}
A_T&\lesssim\|\Vw_T\|_{\VH(\curl)}\|(1-\eta_T)\Vz+\nabla((1-\eta_T)\theta)\|_{\VH(\curl, \{\eta_T\neq 1\})}\\
&\leq \|\Vw_T\|_{\VH(\curl)}\,(1+H)\,\|\Vw\|_{\VH(\curl, \UN^3(\{\eta_T\neq 1\}))}.
\end{align*}
The term $B_T$ can be split as
\begin{align*}
B_T\leq |\CB(\Vw_T,(\id-\pi_H^E)(\eta_T\Vz+\nabla(\eta_T\theta)))|+|\CB(\Vw_T,\pi_H^E(\eta_T\Vz+\nabla(\eta_T\theta)))|.
\end{align*}
Denoting $\Vphi:=(\id-\pi_H^E)(\eta_T\Vz+\nabla(\eta_T\theta))$, we observe $\Vphi\in \VW$ and $\supp\Vphi\subset\Omega\setminus \UN^m(T)$.
Because $\Vphi\in\VW$ with support outside $T$, we have $\CB(\Gf(\mathbf{F}_T),\Vphi)=\mathbf{F}_T(\Vphi)=0$.
Since $\Vphi$ has support outside $\UN^m(T)=\Omega_T$, but $\Gf_{T,m}(\mathbf{F}_T)\in \VW(\Omega_T)$, we also have $\CB(\Gf_{T,m}(\mathbf{F}_T),\Vphi)=0$.
All in all, this means $\CB(\Vw_T , \Vphi )=0$.
Using the stability of $\pi_H^E$ \eqref{eq:stabilityL2}, \eqref{eq:stabilitycurl} and the regular decomposition \eqref{eq:regulardecomp}, we obtain
\begin{align*}
B_T &\leq |\CB(\Vw_T,\pi_H^E(\eta_T\Vz+\nabla(\eta_T\theta)))|\\*
&\lesssim\|\Vw_T\|_{\VH(\curl)}\bigl(\|\eta_T\Vz+\nabla(\eta_T\theta)\|_{L^2(\UN^2(\{\eta_T\neq 1\}))}+(1+H)\|\curl(\eta_T\Vz)\|_{L^2(\UN^2(\{\eta_T\neq 1\}))}\bigr)\\
&\lesssim \|\Vw_T\|_{\VH(\curl)}(1+H)\,\|\Vw\|_{\VH(\curl, \UN^5(\{\eta_T\neq 1\}))}.
\end{align*}
Combining the estimates for $A_T$ and $B_T$
and observing that $\{\eta_T\neq 1\} =\UN^{m+2}(T)$, we deduce
\begin{align*}
\alpha\|\Vw\|_{\VH(\curl)}^2&\lesssim \sum_{T\in\CT_H}\|\Vw_T\|_{\VH(\curl)}\,\|\Vw\|_{\VH(\curl, \UN^{m+7}(T))}
\lesssim \sqrt{C_{\ol, m}}\, \|\Vw\|_{\VH(\curl)}\sqrt{\sum_{T\in\CT_H}\|\Vw_T\|_{\VH(\curl)}^2}.
\end{align*}
Combination with estimate \eqref{eq:errorcorrectorlocal} finishes the proof of \eqref{eq:errorcorrector}. Finally, estimate \eqref{eq:errorcorrector-2} follows with
\begin{align*}
\|\Vw\|_{\VH(\Div)^{\prime}} \leq C H \|\Vw\|_{\VH(\curl)}.
\end{align*}
\end{proof}

\smallskip
\textbf{Changes for the fully discrete localized method.}\hspace{2pt}
Let us briefly consider the fully-discrete setting described in Section \ref{subsec:discreteLOD}. Here we note that, up to a modification of the constants, Theorem \ref{thm:errorcorrectors} also holds for the difference 
$(\Gf_h - \Gf_{h,m})(\mathbf{F})$, where $\Gf_h(\mathbf{F})$ is the Galerkin approximation of $\Gf(\mathbf{F})$ in the discrete space $\VW_h:=\{\Vv_h\in\mathring{\CN}(\CT_h)|\pi_H^E(\Vv_h)=0\}$ and where $\Gf_{h,m}(\mathbf{F})$ is defined analogously to $\Gf_{h,m}(\mathbf{F})$ but where $\VW_h(\Omega_T):=\{ \Vw_h \in \VW_h| \hspace{3pt} \Vw_h \equiv 0 \mbox{ in } \Omega \setminus \Omega_T \}$ replaces $\VW(\Omega_T)$ in the local problems.
Again, the central observation is a decay result similar to Proposition \ref{prop:decaycorrector1}, but now for $\Gf_{h}(\mathbf{F}_T)$.
A few modifications to the proof have to be made, though: The product of the cut-off function $\eta$ and the regular decomposition $\Vz+\nabla\theta$ does not lie in $\mathring{\CN}(\CT_h)$.
Therefore, an additional interpolation operator into $\mathring{\CN}(\CT_H)$ has to be applied.
Here it is tempting to just use the nodal interpolation operator and its stability on piece-wise polynomials, since $\eta \hspace{2pt} \Gf_{h}(\mathbf{F}_T)$ is a piece-wise (quadratic) polynomial. However, the regular decomposition employed is no longer piece-wise polynomial and we hence have to use the Falk-Winther operator $\pi_h^E$ onto the fine space $\mathring{\CN}(\CT_h)$ here.
This means that we have the following modified terms in the proof of Proposition \ref{prop:decaycorrector1}:
\begin{align*}
\tilde{M}_1&:=\Bigl|\bigl(\mu\curl\Vphi, \curl(\id-\pi_H^E)\pi_h^E(\eta\Vz)\bigr)_{L^2(\Omega)}
&&\hspace{-17pt}+\enspace
\bigl(\kappa \Vphi, (\id-\pi_H^E)\pi_h^E(\eta\Vz+\nabla(\eta\theta))\bigr)_{L^2(\Omega)}\Bigr|,
\\
\tilde{M}_2&:=\Bigl|\bigl(\mu\curl\Vphi, \curl\pi_H^ E\pi_h^E(\eta\Vz)\bigr)_{L^2(\Omega)}\Bigr|,
&&
\tilde{M}_3:=\Bigl|\bigl(\kappa \Vphi,\pi_H^E\pi_h^E(\eta\Vz+\nabla(\eta\Vz))\bigr)_{L^2(\Omega)}\Bigr|.
\end{align*}
They can be treated similarly to  
$M_1$, $M_2$ and $M_3$, using in addition the stability of $\pi_h^E$.
Note that the additional interpolation operator $\pi_h^E$ will enlarge the patches slightly, so that we should define $\eta$ via
\begin{align*}
\eta=0\qquad \text{in}\quad \UN^{m-8}(T)\qquad\qquad\qquad\eta=1\qquad\text{in}\quad \Omega\setminus \UN^{m-7}(T).
\end{align*}
The terms $M_4$ and $M_5$ remain unchanged, and we moreover get the terms
\begin{align*}
\tilde{M}_6:=\Bigl|\bigl(\mu\curl\Vphi, \curl(\id-\pi_h^E)(\eta\Vz)\bigr)_{L^2(\Omega)}\Bigr|,
\qquad
\tilde{M}_7:=\Bigl|\bigl(\kappa \Vphi, (\id-\pi_h^E)(\eta\Vz+\nabla(\eta\theta))\bigr)_{L^2(\Omega)}\Bigr|.
\end{align*}
These can be estimated simply using the stability of $\pi_h^E$, the properties of $\eta$ and the regular decomposition \eqref{eq:regulardecomp}.

\section{Falk--Winther interpolation}
\label{sec:intpolimpl}

This section briefly describes the construction of the 
bounded local cochain projection of \cite{FalkWinther2014} for the 
present case of $\VH(\curl)$-problems in three space dimensions.
The two-dimensional case is thoroughly described in the gentle
introductory paper \cite{FalkWinther2015}.
After giving the definition of the operator, we describe how it
can be represented as a matrix. This is important because the 
interpolation operator is part of the algorithm and not a mere
theoretical tool and therefore
required in a practical realization.

\subsection{Definition of the operator}
Let $\Delta_0$ denote the set of vertices of $\CT_H$ and 
let $\mathring{\Delta}_0:=\Delta_0\cap\Omega$ denote the 
interior vertices.
Let $\Delta_1$ denote the set of edges and 
let $\mathring{\Delta}_1$ denote the interior edges, i.e., 
the elements of $\Delta_1$ that
are not a subset of $\partial\Omega$.
The space $\mathring{\CN}(\CT_H)$ is spanned by the well-known edge-oriented
basis $(\Vpsi_E)_{E\in\mathring{\Delta}_1}$ defined for any $E\in\mathring{\Delta}_1$
through the property
\begin{equation*}
\fint_E \Vpsi_E\cdot \Vt_E\,ds = 1
\quad\text{and}\quad
\fint_{E'} \Vpsi_E\cdot \Vt_E\,ds = 0
\quad\text{for all }E'\in\mathring{\Delta}_1\setminus\{E\}.
\end{equation*}
Here $\Vt_E$ denotes the unit tangent to the edge $E$ with a globally
fixed sign.
Any vertex $z\in\Delta_0$ possesses a nodal patch (sometimes also called
macroelement)
\begin{equation*}
\omega_z:=\Int\Big(\bigcup\{T\in\CT_H : z\in T\}\Big).
\end{equation*}
For any edge $E\in\Delta_1$ shared by two vertices
$z_1,z_2\in\Delta_0$ such that
$E=\operatorname{conv}\{z_1,z_2\}$, the extended edge patch
reads
\begin{equation*}
\omega_E^{\mathit{ext}} := \omega_{z_1}\cup\omega_{z_2}. 
\end{equation*}
The restriction of the mesh $\CT_H$ to $\omega_E^{\mathit{ext}}$
is denoted by
$\CT_H(\omega_E^{\mathit{ext}})$.
Let $\CS^1(\CT_H(\omega_E^{\mathit{ext}}))$ denote the (scalar-valued) first-order
Lagrange finite element space with respect to 
$\CT_H(\omega_E^{\mathit{ext}})$ and let 
$\CN(\CT_H(\omega_E^{\mathit{ext}}))$ denote the 
lowest-order N\'ed\'elec finite element space over 
$\CT_H(\omega_E^{\mathit{ext}})$.
The operator 
\[ 
  Q^1_E:
   \VH(\curl, \omega_E^{\mathit{ext}})
   \to
   \CN(\CT_H(\omega_E^{\mathit{ext}}))
\]
is defined for any $\Vu\in \VH(\curl, \omega_E^{\mathit{ext}})$
via
\begin{equation*}
\begin{aligned}
 (\Vu-Q^1_E \Vu, \nabla \tau) &= 0 \quad 
    &&\text{for all } \tau\in \CS^1(\CT_H(\omega_E^{\mathit{ext}}))
 \\
 (\curl (\Vu-Q^1_E \Vu),\curl \Vv) &=0
    &&\text{for all } \Vv\in \CN(\CT_H(\omega_E^{\mathit{ext}})).
\end{aligned}
\end{equation*}

Given any vertex $y\in\Delta_0$, define the piecewise constant function
$z^0_y$ by
\begin{equation*}
z^0_y = \begin{cases}
     (\operatorname{meas}(\omega_y))^{-1} &\text{in } \omega_y \\
         0                      &\text{in } \Omega\setminus\omega_y
        \end{cases}
\end{equation*}
Given any edge $E\in\Delta_1$ shared by vertices
$y_1,y_2\in\Delta_0$ such that $E=\operatorname{conv}\{y_1,y_2\}$,
define
\begin{equation*}
(\delta z^0)_E := 
  z^0_{y_2} - z^0_{y_1} .
\end{equation*}
Let $E\in\Delta_1$ and denote by 
$\mathring{\mathcal{RT}}(\CT_H(\omega_E^{\mathit{ext}}))$ the lowest-order
Raviart--Thomas space with respect to 
$\CT_H(\omega_E^{\mathit{ext}})$ with vanishing normal trace on
the boundary $\partial (\omega_E^{\mathit{ext}})$.
Let for any $E\in\Delta_1$
the field
$\Vz_E^1\in\mathring{\mathcal{RT}}(\CT_H(\omega_E^{\mathit{ext}}))$
be defined by
\begin{equation*}
\begin{aligned}
  \Div \Vz_E^1 &=-(\delta z^0)_E \quad &&
  \\
  (\Vz_E^1,\curl\Vtau) &= 0 
     &&\text{for all } 
      \Vtau\in\mathring{\CN}(\CT_H(\omega_E^{\mathit{ext}}))
\end{aligned}
\end{equation*}
where $\mathring{\CN}(\CT_H(\omega_E^{\mathit{ext}}))$ denotes
the N\'ed\'elec finite element functions over 
$\CT_H(\omega_E^{\mathit{ext}})$ with vanishing tangential trace 
on the boundary $\partial(\omega_E^{\mathit{ext}})$.
The operator 
$M^1:L^2(\Omega;\cz^3)\to\mathring{\CN}(\CT_H)$ maps any
$\Vu\in L^2(\Omega;\cz^3)$ to
\begin{equation*}
M^1\Vu :=
\sum_{E\in\mathring{\Delta}_1}
     (\operatorname{length}(E))^{-1}
      \int_{\omega_E^{\mathit{ext}}} \Vu\cdot \Vz_E^1\,dx\, \Vpsi_E.
\end{equation*}

The operator 
\[
 Q^1_{y,-} : \VH(\curl,\omega_E^{\mathit{ext}})
 \to 
 \CS^1(\CT_H(\omega_E^{\mathit{ext}}))
\] 
is the solution operator of a local discrete Neumann problem.
For any $\Vu\in \VH(\curl, \omega_E^{\mathit{ext}})$,
the function
$ Q^1_{y,-} \Vu $ solves
\begin{equation*}
\begin{aligned}
(\Vu-\nabla Q^1_{y,-} \Vu,\nabla v) &= 0
   \quad&&\text{for all } v\in \CS^1(\CT_H(\omega_E^{\mathit{ext}}))
\\
\int_{\omega_E^{\mathit{ext}}} Q^1_{y,-} \Vu\,dx & = 0. &&
\end{aligned}
\end{equation*}
Define now the operator 
$S^1:\VH_0(\curl,\Omega)\to \mathring{\CN}(\CT_H)$
via
\begin{equation}\label{e:S1def1}
S^1 \Vu :=
M^1 \Vu +
\sum_{y\in\mathring{\Delta}_0} 
  (Q^1_{y,-}\Vu)(y)\nabla \lambda_y .
\end{equation}
The second sum on the right-hand side can be rewritten in terms
of the basis functions $\Vpsi_E$.
The inclusion
$\nabla \mathring{\CS}^1(\CT_H)\subseteq \mathring{\CN}(\CT_H)$
follows from the principles of finite element exterior calculus
\cite{ArnoldFalkWinther2006,ArnoldFalkWinther2010}.
Given an interior vertex 
$y\in\mathring{\Delta}_0$, the expansion in terms of the basis
$(\Vpsi_E)_{E\in\mathring{\Delta}_1}$ reads
\begin{equation*}
\nabla\lambda_z
=
\sum_{E\in\mathring{\Delta}_1} \fint_E \nabla\lambda_z\cdot \Vt_E\,ds\,\Vpsi_E
=
\sum_{E\in\Delta_1(z)} 
  \frac{\operatorname{sign}(\Vt_E\cdot\nabla\lambda_z)}{\operatorname{length}(E)}
   \Vpsi_E
\end{equation*}
where $\Delta_1(z)\subseteq\mathring{\Delta}_1$
is the set of all edges that contain $z$.
Thus, $S^1$ from \eqref{e:S1def1} can be rewritten as
\begin{equation}\label{e:S1def2}
S^1 \Vu :=
M^1 \Vu +
\sum_{E\in\mathring{\Delta}_1}
   (\operatorname{length}(E))^{-1}
\big((Q^1_{y_2(E),-}\Vu)(y_2(E)) - (Q^1_{y_1(E),-}\Vu)(y_1(E))\big) 
\Vpsi_E
\end{equation}
where $y_1(E)$ and $y_2(E)$ denote the endpoints of $E$
(with the orientation convention 
$\Vt_E = (y_2(E)-y_1(E))/\operatorname{length}(E)$).
Finally, the Falk-Winter interpolation operator $\pi_H^E:\VH_0(\curl, \Omega)\to\mathring{\CN}(\CT_H)$ is defined as
\begin{equation}\label{e:R1def}
\pi_H^E \Vu
:=
S^1 \Vu
+
\sum_{E\in\mathring{\Delta}_1}
   \fint_E
    \big((\id-S^1)Q^1_E \Vu\big)\cdot \Vt_E\,ds
   \,\Vpsi_E .
\end{equation}

\subsection{Algorithmic aspects}

Given a mesh $\CT_H$ and a refinement $\CT_h$, the linear 
projection
$\pi_H : \mathring{\CN}(\CT_h)\to \mathring{\CN}(\CT_H)$ can be represented by a matrix
$\mathsf{P}\in\mathbb R^{\dim \mathring{\CN}(\CT_H)\times\dim \mathring{\CN}(\CT_h)}$.
This subsection briefly sketches the assembling of that matrix.
The procedure involves the solution of local discrete problems
on the macroelements. It is important to note that these problems
are of small size because the mesh $\CT_h$ is a refinement
of $\CT_H$. 

Given an interior edge $E\in\mathring{\Delta}_1^H$
of $\CT_H$ and an interior edge 
$e\in\mathring{\Delta}_1^h$ of $\CT_h$, the interpolation $\pi_H \Vpsi_e$
has an expansion
\begin{equation*}
 \pi_H \Vpsi_e= \sum_{E'\in\mathring{\Delta}_1^H} c_{E'} \Vpsi_{E'}
\end{equation*}
for real coefficients $(c_{E'})_{E'\in\mathring{\Delta}_1^H}$.
The coefficient $c_E$ is zero whenever $e$ is not contained in the 
closure of the extended edge patch $\overline{\omega}_E^{\mathit{ext}}$.
The assembling can therefore be organized in a loop over all interior
edges in $\mathring{\Delta}_1^H$.
Given a global numbering of the edges in 
$\mathring{\Delta}_1^H$, 
each edge $E\in\mathring{\Delta}_1^H$
is equipped with a unique index 
$I_H(E)\in\{1,\dots,\operatorname{card}(\mathring{\Delta}_1^H)\}$.
Similarly, the numbering of edges in $\mathring{\Delta}_1^h$ is denoted
by $I_h$.

The matrix $\mathsf{P}=\mathsf{P_1}+\mathsf{P_2}$ will be composed 
as the sum of matrices
$\mathsf{P_1}$, $\mathsf{P_2}$ that represent the two summands on the 
right-hand side of \eqref{e:R1def}.
Those will be assembled in loops over the interior edges.
Matrices $\mathsf{P_1}$, $\mathsf{P_2}$ are initialized as
empty sparse matrices.

\subsubsection{Operator $\mathsf{P_1}$}

\noindent
\textbf{for} $E\in\mathring{\Delta}_1^H$ \textbf{do}

Let the interior edges in 
$\mathring{\Delta}_1^h$ that lie inside
$\overline{\omega}_E^{\mathit{ext}}$
be denoted with
$\{e_1,e_2,\dots,e_N\}$ for some $N\in\mathbb N$.
The entries $\mathsf{P}_1(I_H(E),[I_h(e_1)\dots I_h(e_N)])$ of the matrix
$\mathsf{P}_1$ are now determined as follows.
Compute $\Vz^1_E \in \mathring{\mathcal{RT}}(\CT_H({\omega}_E^{\mathit{ext}}))$.
The matrix
$\mathsf{M}_E\in\mathbb R^{1\times N}$ defined via
\[
  \mathsf{M}_E 
 :=
 (\operatorname{length}(E))^{-1}
 \left[
 \int_{{\omega}_E^{\mathit{ext}}} \Vz^1_E\cdot\Vpsi_{e_j}\,dx
 \right]_{j=1}^N
\]
represents the map of the basis functions on the fine mesh
to the coefficient of $M^1$ contributing to $\Vpsi_E$ on the coarse mesh.
Denote by 
$\mathsf{A}_{y_j(E)}$ and 
$\mathsf{B}_{y_j(E)}$ ($j=1,2$)
the stiffness and right-hand side matrix
representing the system for the operator $Q_{y_j(E),-}$
\begin{align*}
 \mathsf{A}_{y_j(E)}
 &:=
 \left[
 \int_{\omega_{y_j(E)}} \nabla \phi_y \cdot\nabla \phi_z\,dx
 \right]_{y,z\in\Delta_0(\CT_H(\omega_{y_j(E)}))},
\\
 \mathsf{B}_{y_j(E)}
 &:=
 \left[
 \int_{\omega_{y_j(E)}} \nabla \phi_y \cdot\Vpsi_{e_j}\,dx
 \right]_{\substack{y\in\Delta_0(\CT_H(\omega_{y_j(E)}))\\ j=1,\dots,N}}.
\end{align*}
After enhancing the system 
to
$\tilde{\mathsf{A}}_{y_j(E)}$ and 
$\tilde{\mathsf{B}}_{y_j(E)}$ 
(with a Lagrange multiplier accounting for the mean constraint),
it is uniquely solvable.
Set 
$\tilde{\mathsf{Q}}_{y_j(E)} =
  \tilde{\mathsf{A}}_{y_j(E)}^{-1}\tilde{\mathsf{B}}_{y_j(E)}$
and extract the row corresponding to the vertex $y_j(E)$
\[
  \mathsf{Q}_j:=
     (\operatorname{length}(E))^{-1}
  \tilde{\mathsf{Q}}_{y_j(E)}[y_j(E),:]
  \in \mathbb R^{1\times N}.
\]
Set
\[ 
  \mathsf{P}_1(I_H(E),[I_h(e_1)\dots I_h(e_N)])
  =
  \mathsf{M}_E + \mathsf{Q}_1 -\mathsf{Q}_2 .
\]
\noindent
\textbf{end}

\subsubsection{Operator $\mathsf{P_2}$}

\noindent
\textbf{for} $E\in\mathring{\Delta}_1^H$ \textbf{do}

Denote the matrices
-- where indices $j,k$ run from $1$ to
$\operatorname{card}(\Delta_1(\CT_H({\omega}_E^{\mathit{ext}})))$, $y$ through \linebreak[4]$\Delta_0(\CT_H({\omega}_E^{\mathit{ext}}))$, and $\ell=1,\ldots, N$ --
\begin{equation*}
\mathsf{S}_E
 :=
  \left[
 \int_{{\omega}_E^{\mathit{ext}}} 
       \curl \Vpsi_{E_j} \cdot\curl\Vpsi_{E_k}\,dx
 \right]_{j,k}
\mathsf{T}_E
 :=
 \left[
 \int_{{\omega}_E^{\mathit{ext}}} 
       \Vpsi_{E_j} \cdot\nabla\lambda_{y}\,dx
 \right]_{j,y}
\end{equation*}
and
\begin{equation*}
\mathsf{F}_E
 :=
  \left[
 \int_{{\omega}_E^{\mathit{ext}}} 
       \curl \Vpsi_{E_j} \cdot\curl\Vpsi_{e_\ell}\,dx
 \right]_{j,\ell}
\mathsf{G}_E
 :=
 \left[
 \int_{{\omega}_E^{\mathit{ext}}} 
        \Vpsi_{e_\ell} \cdot\nabla\lambda_{y}\,dx
 \right]_{y, \ell} .
\end{equation*}
Solve the saddle-point system
\begin{equation*}
 \begin{bmatrix}
  \mathsf{S} & \mathsf{T}^* \\ \mathsf{T} & 0
 \end{bmatrix}
\begin{bmatrix}
  \mathsf{U} \\ \mathsf{V}
 \end{bmatrix}
=
\begin{bmatrix}
  \mathsf{F} \\ \mathsf{G}
 \end{bmatrix} .
\end{equation*}
(This requires an additional one-dimensional gauge condition
because the sum of the test functions $\sum_y\nabla\lambda_y$ equals
zero.)
Assemble the operator $S^1$ (locally) as described in the 
previous step and denote this matrix by
$\mathsf{P}_1^{\mathit{loc}}$.
Compute $\mathsf{U}- \mathsf{P}_1^{\mathit{loc}} \mathsf{U}$ 
and extract the line 
$\mathsf{X}$ corresponding to the edge $E$

\[ 
  \mathsf{P_2}(I_H(E),[I_h(e_1)\dots I_h(e_N)])
  =
  \mathsf{X} .
\]
\noindent
\textbf{end}

\section*{Conclusion}
In this paper, we suggested a procedure for the numerical homogenization of $\VH(\curl)$-elliptic problems.
The exact solution is decomposed into a coarse part, which is a good approximation in $\VH(\Div)^\prime$, and a corrector contribution by using the Falk-Winther interpolation operator.
We showed that this decomposition gives an optimal order approximation in $\VH(\curl)$, independent of the regularity of the exact solution.
Furthermore, the corrector operator can be localized to patches of macro elements, which allows for an efficient computation.
This results in a generalized finite element method in the spirit of the Localized Orthogonal Decomposition which utilizes
the bounded local cochain projection of the Falk-Winther
as part of the algorithm.

\section*{Acknowledgments}
Main parts of this paper were written while the authors enjoyed the kind hospitality of the Hausdorff Institute for Mathematics in Bonn.
PH and BV acknowledge financial support by the DFG in the project OH 98/6-1 ``Wave propagation in periodic structures and negative refraction mechanisms''.
DG acknowledges support by the DFG through CRC 1173
``Wave phenomena: analysis and numerics'' and by the 
Baden-W\"urttemberg Stiftung
(Eliteprogramm f\"ur Postdocs)
through the project
``Mehrskalenmethoden für Wellenausbreitung in heterogenen Materialien und
 Metamaterialien''.


\begin{thebibliography}{10}

\bibitem{AH17LODwaves}
A.~Abdulle and P.~Henning.
\newblock Localized orthogonal decomposition method for the wave equation with
  a continuum of scales.
\newblock {\em Math. Comp.}, 86(304):549--587, 2017.

\bibitem{ArnoldFalkWinther2006}
D.~N. Arnold, R.~S. Falk, and R.~Winther.
\newblock Finite element exterior calculus, homological techniques, and
  applications.
\newblock {\em Acta Numer.}, 15:1--155, 2006.

\bibitem{ArnoldFalkWinther2010}
D.~N. Arnold, R.~S. Falk, and R.~Winther.
\newblock Finite element exterior calculus: from {H}odge theory to numerical
  stability.
\newblock {\em Bull. Amer. Math. Soc. (N.S.)}, 47(2):281--354, 2010.

\bibitem{Bab70fem}
I.~Babu{\v{s}}ka.
\newblock Error-bounds for finite element method.
\newblock {\em Numer. Math.}, 16:322--333, 1971.

\bibitem{BGL13regularitymaxwell}
A.~Bonito, J.-L. Guermond, and F.~Luddens.
\newblock Regularity of the {M}axwell equations in heterogeneous media and
  {L}ipschitz domains.
\newblock {\em J. Math. Anal. Appl.}, 408(2):498--512, 2013.

\bibitem{BrG16}
D.~Brown and D.~Gallistl.
\newblock Multiscale sub-grid correction method for time-harmonic
  high-frequency elastodynamics with wavenumber explicit bounds.
\newblock {\em ArXiv e-print 1608.04243}, 2016.

\bibitem{bgp2017}
D.~Brown, D.~Gallistl, and D.~Peterseim.
\newblock Multiscale {P}etrov-{G}alerkin method for high-frequency
  heterogeneous {H}elmholtz equations.
\newblock In M.~Griebel and M.~A. Schweitzer, editors, {\em Meshfree Methods
  for Partial Differential Equations VII}, volume 115 of {\em Lect. Notes
  Comput. Sci. Eng.}, pages 85--115. 2017.

\bibitem{CZAL10maxwell}
L.~Cao, Y.~Zhang, W.~Allegretto, and Y.~Lin.
\newblock Multiscale asymptotic method for {M}axwell's equations in composite
  materials.
\newblock {\em SIAM J. Numer. Anal.}, 47(6):4257--4289, 2010.

\bibitem{Chr07intpol}
S.~H. Christiansen.
\newblock Stability of {H}odge decompositions in finite element spaces of
  differential forms in arbitrary dimension.
\newblock {\em Numer. Math.}, 107(1):87--106, 2007.

\bibitem{CW08intpol}
S.~H. Christiansen and R.~Winther.
\newblock Smoothed projections in finite element exterior calculus.
\newblock {\em Math. Comp.}, 77(262):813--829, 2008.

\bibitem{CFS17hmmmaxwell}
P.~Ciarlet~Jr., S.~Fliss, and C.~Stohrer.
\newblock On the approximation of electromagnetic fields by edge finite
  elements. {P}art 2: {A} heterogeneous multiscale method for {M}axwell's
  equations.
\newblock {\em Comput. Math. Appl.}, 73(9):1900--1919, 2017.

\bibitem{Cost90regmaxwellremark}
M.~Costabel.
\newblock A remark on the regularity of solutions of {M}axwell's equations on
  {L}ipschitz domains.
\newblock {\em Math. Methods Appl. Sci.}, 12(4):365--368, 1990.

\bibitem{CDN99maxwellinterface}
M.~Costabel, M.~Dauge, and S.~Nicaise.
\newblock Singularities of {M}axwell interface problems.
\newblock {\em M2AN Math. Model. Numer. Anal.}, 33(3):627--649, 1999.

\bibitem{DB05maxwellpintpol}
L.~Demkowicz and A.~Buffa.
\newblock {$H^1$}, {$H(\textrm{curl})$} and {$H(\textrm{div})$}-conforming
  projection-based interpolation in three dimensions. {Q}uasi-optimal
  {$p$}-interpolation estimates.
\newblock {\em Comput. Methods Appl. Mech. Engrg.}, 194(2-5):267--296, 2005.

\bibitem{DH14aposteriorimaxwell}
A.~Demlow and A.~N. Hirani.
\newblock A posteriori error estimates for finite element exterior calculus:
  the de {R}ham complex.
\newblock {\em Found. Comput. Math.}, 14(6):1337--1371, 2014.

\bibitem{EP04negphC}
A.~Efros and A.~Pokrovsky.
\newblock Dielectroc photonic crystal as medium with negative electric
  permittivity and magnetic permeability.
\newblock {\em Solid State Communications}, 129(10):643--647, 2004.

\bibitem{EHMP16LODimpl}
C.~{Engwer}, P.~{Henning}, A.~{M{\aa}lqvist}, and D.~{Peterseim}.
\newblock {Efficient implementation of the Localized Orthogonal Decomposition
  method}.
\newblock {\em ArXiv e-prints 1602.01658}, 2016.

\bibitem{EG15intpolbestapprox}
A.~{Ern} and J.-L. {Guermond}.
\newblock {Finite element quasi-interpolation and best approximation}.
\newblock {\em ArXiv e-prints 1505.06931}, 2015.

\bibitem{EG15intpol}
A.~{Ern} and J.-L. {Guermond}.
\newblock {Mollification in strongly Lipschitz domains with application to
  continuous and discrete De Rham complex}.
\newblock {\em ArXiv e-prints 1509.01325}, 2015.

\bibitem{FalkWinther2014}
R.~S. Falk and R.~Winther.
\newblock Local bounded cochain projections.
\newblock {\em Math. Comp.}, 83(290):2631--2656, 2014.

\bibitem{FalkWinther2015}
R.~S. Falk and R.~Winther.
\newblock Double complexes and local cochain projections.
\newblock {\em Numer. Methods Partial Differential Equations}, 31(2):541--551,
  2015.

\bibitem{FR05maxwell}
P.~Fernandes and M.~Raffetto.
\newblock Existence, uniqueness and finite element approximation of the
  solution of time-harmonic electromagnetic boundary value problems involving
  metamaterials.
\newblock {\em COMPEL}, 24(4):1450--1469, 2005.

\bibitem{GP15scatteringPG}
D.~Gallistl and D.~Peterseim.
\newblock Stable multiscale {P}etrov-{G}alerkin finite element method for high
  frequency acoustic scattering.
\newblock {\em Comput. Methods Appl. Mech. Engrg.}, 295:1--17, 2015.

\bibitem{HHM16LODmixed}
F.~Hellman, P.~Henning, and A.~M{\aa}lqvist.
\newblock Multiscale mixed finite elements.
\newblock {\em Discrete Contin. Dyn. Syst. Ser. S}, 9(5):1269--1298, 2016.

\bibitem{HM14LODbdry}
P.~Henning and A.~M{\aa}lqvist.
\newblock Localized orthogonal decomposition techniques for boundary value
  problems.
\newblock {\em SIAM J. Sci. Comput.}, 36(4):A1609--A1634, 2014.

\bibitem{HOV15maxwellHMM}
P.~Henning, M.~Ohlberger, and B.~Verf{\"u}rth.
\newblock A new {H}eterogeneous {M}ultiscale {M}ethod for time-harmonic
  {M}axwell's equations.
\newblock {\em SIAM J. Numer. Anal.}, 54(6):3493--3522, 2016.

\bibitem{HP13oversampl}
P.~Henning and D.~Peterseim.
\newblock Oversampling for the multiscale finite element method.
\newblock {\em Multiscale Model. Simul.}, 11(4):1149--1175, 2013.

\bibitem{Hipt02FEem}
R.~Hiptmair.
\newblock Finite elements in computational electromagnetism.
\newblock {\em Acta Numer.}, 11:237--339, 2002.

\bibitem{JJWM08phc}
J.~Joannopolous, S.~G. Johnson, J.~Winn, and R.~Meade.
\newblock {\em Phtonic crystals: {M}olding the flow of light}.
\newblock Princeton Univ. Press, 2nd edition, 2008.

\bibitem{LS15negindex}
A.~{Lamacz} and B.~{Schweizer}.
\newblock A negative index meta-material for {M}axwell's equations.
\newblock {\em SIAM J. Math. Anal.}, 48(6):4155--4174, 2016.

\bibitem{CJJP02negrefraction}
C.~Luo, S.~G. Johnson, J.~Joannopolous, and J.~Pendry.
\newblock All-angle negative refraction without negative effective index.
\newblock {\em Phys. Rev. B}, 65(2001104), May 2002.

\bibitem{MP14LOD}
A.~M{\aa}lqvist and D.~Peterseim.
\newblock Localization of elliptic multiscale problems.
\newblock {\em Math. Comp.}, 83(290):2583--2603, 2014.

\bibitem{Monk}
P.~Monk.
\newblock {\em Finite element methods for {M}axwell's equations}.
\newblock Numerical Mathematics and Scientific Computation. Oxford University
  Press, New York, 2003.

\bibitem{OV16a}
M.~Ohlberger and B.~Verf{\"u}rth.
\newblock Localized orthogonal decomposition for two-scale {H}elmholtz-type
  problems.
\newblock {\em arXiv e-print 1605.03410}, 2016.

\bibitem{PZ02Schwarz}
J.~E. Pasciak and J.~Zhao.
\newblock Overlapping {S}chwarz methods in {$H$}(curl) on polyhedral domains.
\newblock {\em J. Numer. Math.}, 10(3):221--234, 2002.

\bibitem{P15LODhelmholtz}
D.~{Peterseim}.
\newblock Eliminating the pollution effect in {H}elmholtz problems by local
  subscale correction.
\newblock {\em Math. Comp.}, 86, 2017.

\bibitem{PeS17}
D.~Peterseim and M.~Schedensack.
\newblock Relaxing the {CFL} condition for the wave equation on adaptive
  meshes.
\newblock {\em J. Sci. Comput. (Online First)}, 2017.

\bibitem{Sch05multilevel}
J.~Sch{\"o}berl.
\newblock {A} multilevel decomposition result in ${H}(curl)$.
\newblock In P.~H. P.~Wesseling, C.W.~Oosterlee, editor, {\em Multigrid,
  Multilevel and Multiscale Methods, Proceedings of the 8th European Multigrid
  Conference, EMG}, 2005.

\bibitem{Sch08aposteriori}
J.~Sch{\"o}berl.
\newblock A posteriori error estimates for {M}axwell equations.
\newblock {\em Math. Comp.}, 77(262):633--649, 2008.

\bibitem{CH15homerrormaxwell}
V.~Tiep~Chu and V.~H. Hoang.
\newblock Homogenization error for two scale {M}axwell equations.
\newblock {\em arXiv e-print 1512.02788}, 2015.

\bibitem{Well2}
N.~Wellander and G.~Kristensson.
\newblock Homogenization of the {M}axwell equations at fixed frequency.
\newblock {\em SIAM J. Appl. Math.}, 64(1):170--195, 2003.

\end{thebibliography}
\end{document}